\documentclass[draft]{amsart}

\usepackage{amssymb}
\usepackage[all]{xy}

\newtheorem{thm}{Theorem}[section]
\newtheorem{cor}[thm]{Corollary}
\newtheorem{prop}[thm]{Proposition}
\newtheorem{lem}[thm]{Lemma}
\newtheorem{claim}[thm]{Claim}

\theoremstyle{definition}
\newtheorem{dfn}[thm]{Definition}
\newtheorem{ex}[thm]{Example}

\newtheorem{const}[thm]{Construction}
\newtheorem{condition}[thm]{Condition}

\theoremstyle{remark}
\newtheorem{rem}[thm]{Remark}
\newtheorem{caution}[thm]{Caution}

\numberwithin{equation}{section}

\begin{document}

\title[Frobenius condition on a pretriangulated category]{Frobenius condition on a pretriangulated category, and triangulation on the associated stable category}

\author{Hiroyuki NAKAOKA}
\address{Graduate School of Mathematical Sciences, The University of Tokyo 
3-8-1 Komaba, Meguro, Tokyo, 153-8914 Japan}

\email[Hiroyuki NAKAOKA]{deutsche@ms.u-tokyo.ac.jp}

\thanks{The author wishes to thank Professor Toshiyuki Katsura for his encouragement}
\thanks{The author wishes to thank Professor Kiriko Kato and Professor Osamu Iyama for their useful comments and advices.}

\begin{abstract}
As shown by Happel, from any Frobenius exact category, we can construct a triangulated category as a stable category. On the other hand, it was shown by Iyama and Yoshino that if a pair of subcategories $\mathcal{D}\subseteq\mathcal{Z}$ in a triangulated category satisfies certain conditions (i.e., $(\mathcal{Z},\mathcal{Z})$ is a $\mathcal{D}$-mutation pair), then $\mathcal{Z}/\mathcal{D}$ becomes a triangulated category.
In this article, we consider a simultaneous generalization of these two constructions.
\end{abstract}

\maketitle

\section{Introduction and Preliminaries}
Throughout this article, we fix an additive category $\mathcal{C}$.
Any subcategory of $\mathcal{C}$ will be assumed to be full, additive and replete.
A subcategory is called {\it replete} if it is closed under isomorphisms.

When we say $\mathcal{Z}$ is an exact category, we only consider an extension-closed subcategory of an abelian category.

For any category $\mathcal{K}$, we write abbreviately $K\in\mathcal{K}$, to indicate that $K$ is an object of $\mathcal{K}$.
For any $K,L\in\mathcal{K}$, let $\mathcal{K}(K,L)$ denote the set of morphisms from $K$ to $L$.
If $\mathcal{M},\mathcal{N}$ are full subcategories of $\mathcal{K}$, then $\mathcal{K}(\mathcal{M},\mathcal{N})=0$ means that $\mathcal{K}(M,N)=0$ for any $M\in\mathcal{M}$ and $N\in\mathcal{N}$.
Similarly, $\mathcal{K}(K,\mathcal{N})=0$ means $\mathcal{K}(K,N)=0$ for any $N\in\mathcal{N}$.

If $\mathcal{K}$ is an additive category and $\mathcal{L}$ is a full additive replete subcategory which is closed under finite direct summands, then $\mathcal{K}/\mathcal{L}$ denotes the quotient category of $\mathcal{K}$ by the ideal generated by $\mathcal{L}$. The image of $f\in\mathcal{K}(X,Y)$ will be denoted by $\underline{f}\in\mathcal{K}/\mathcal{L}(X,Y)$.

\bigskip

As shown by Happel \cite{H}, If we are given a Frobenius exact category $\mathcal{E}$, then the stable category $\mathcal{E}/\mathcal{I}$, where $\mathcal{I}$ is the full subcategory of injectives, carries a structure of a triangulated category. 

On the other hand, it was shown by Iyama and Yoshino that if $\mathcal{D}\subseteq\mathcal{Z}$ is a pair of subcategories in a triangulated category $\mathcal{C}$ such that $(\mathcal{Z},\mathcal{Z})$ is a $\mathcal{D}$-mutation pair, then the quotient category $\mathcal{Z}/\mathcal{D}$ becomes a triangulated category.
By definition, $(\mathcal{Z},\mathcal{Z})$ is a $\mathcal{D}$-mutation pair if it satisfies
\begin{enumerate}
\item $\mathcal{C}(\mathcal{Z},\mathcal{D}[1])=\mathcal{C}(\mathcal{D},\mathcal{Z}[1])=0$,
\item For any object $X\in\mathcal{Z}$, there exists a distinguished triangle
\[ X\rightarrow D\rightarrow Z\rightarrow \Sigma X \]
with $D\in\mathcal{D}$ and $Z\in\mathcal{Z}$,
\item For any object $Z\in\mathcal{Z}$, there exists a distinguished triangle
\[ X\rightarrow D\rightarrow Z\rightarrow \Sigma X \]
with $X\in\mathcal{Z}$ and $D\in\mathcal{D}$.
\end{enumerate}

In this article, we make a simultaneous generalization of these two constructions, by using a slight modification of a {\it pretriangulated category} in \cite{BR}. To emphasize this modification, we call it a \lq pseudo-'triangulated category. As in Definition \ref{DefPreTr}, a pseudo-triangulated category is an additive category $\mathcal{C}$ with a {\it pseudo-triangulation} $(\Sigma,\Omega,\triangleright,\triangleleft,\psi)$.

As in Example \ref{ExTwoCases}, a pseudo-triangulated category $\mathcal{C}$ is abelian if and only if $\Sigma=\Omega=0$, and $\mathcal{C}$ is triangulated if and only if $\Sigma\cong\Omega^{-1}$. An {\it extension} in $\mathcal{C}$ is a simultaneous generalization of a short exact sequence in the abelian case, and a distinguished triangle in the triangulated case (Definition \ref{DefExtension}).
For an extension-closed subcategory $\mathcal{Z}\subseteq\mathcal{C}$, we define the {\it Frobenius condition} on it (Definition \ref{DefFrobCond}). This is equivalent to the ordinary Frobenius condition in the case of $\Sigma=\Omega=0$, and related to the existence of a mutation pair in the triangulated case (Example \ref{ExFrobCond} and Corollary \ref{CorFrobenius}). As a main theorem, in Theorem \ref{MainThm}, we show if $\mathcal{Z}$ is Frobenius, then the associated stable category becomes a triangulated category. In the above two cases, this recovers the Happel's and Iyama-Yoshino's constructions, respectively.

\begin{center}
\begin{tabular}
[c]{|c|c|c|}\hline
& $\Sigma=\Omega=0$ & $\Sigma\cong\Omega^{-1}$ \\\hline
Pretriangulated & abelian & triangulated \\\hline
Extension & short exact sequence & distinguished triangle \\\hline
Frobenius condition & Frobenius condition & Corollary \ref{CorFrobenius} \\\hline
Theorem \ref{MainThm} & Happel's construction & Iyama-Yoshino's construction \\\hline
\end{tabular}
\end{center}

\section{One-sided triangulated categories}

\begin{dfn}[right triangluation cf. \cite{BM}, \cite{BR}]\label{DefRTR}
Let $\Sigma\colon\mathcal{C}\rightarrow\mathcal{C}$ be an additive endofunctor, and let $\mathcal{RT}(\mathcal{C},\Sigma)$ be the category of diagrams of the form
\[ A\overset{f}{\longrightarrow}B\overset{g}{\longrightarrow}C\overset{h}{\longrightarrow}\Sigma A. \]
A morphism from $A\overset{f}{\longrightarrow}B\overset{g}{\longrightarrow}C\overset{h}{\longrightarrow}\Sigma A$ to $A^{\prime}\overset{f^{\prime}}{\longrightarrow}B^{\prime}\overset{g^{\prime}}{\longrightarrow}C^{\prime}\overset{h^{\prime}}{\longrightarrow}\Sigma A^{\prime}$ is a triplet $(a,b,c)$ of morphisms $a\in\mathcal{C}(A,A^{\prime})$, $b\in\mathcal{C}(B,B^{\prime})$ and $c\in\mathcal{C}(C,C^{\prime})$, satisfying
\[ b\circ f=f^{\prime}\circ a,\quad c\circ g=g^{\prime}\circ b,\quad \Sigma a\circ h=h^{\prime}\circ c. \]

A pair $(\Sigma, \triangleright)$ of $\Sigma$ and a full replete subcategory $\triangleright\subseteq\mathcal{RT}(\mathcal{C},\Sigma)$ is called a {\it right triangulation} on $\mathcal{C}$ if it satisfies the following conditions. 
Remark that $\Sigma$ is not necessarily an equivalence.
\begin{itemize}
\item[(RTR1)]
For any $A\in\mathcal{C}$, $0\rightarrow A\overset{\mathrm{id}_A}{\longrightarrow}A\rightarrow \Sigma0=0$ is in $\triangleright$.
For any morphism $f\in\mathcal{C}(A,B)$, there exists an object $A\overset{f}{\longrightarrow}B\overset{g}{\longrightarrow}C\overset{h}{\longrightarrow}\Sigma A$ in $\triangleright$.

\item[(RTR2)]If $A\overset{f}{\longrightarrow}B\overset{g}{\longrightarrow}C\overset{h}{\longrightarrow}\Sigma A$ is in $\triangleright$, then $B\overset{g}{\longrightarrow}C\overset{h}{\longrightarrow}\Sigma A\overset{-\Sigma f}{\longrightarrow}\Sigma B$ is also in $\triangleright$.

\item[(RTR3)]If we are given two objects $A\overset{f}{\longrightarrow}B\overset{g}{\longrightarrow}C\overset{h}{\longrightarrow}\Sigma A$ and $A^{\prime}\overset{f^{\prime}}{\longrightarrow}B^{\prime}\overset{g^{\prime}}{\longrightarrow}C^{\prime}\overset{h^{\prime}}{\longrightarrow}\Sigma A^{\prime}$ in $\triangleright$ and two morphisms $a\in\mathcal{C}(A,A^{\prime})$ and $b\in\mathcal{C}(B,B^{\prime})$ satisfying $b\circ f=f^{\prime}\circ a$, then there exists $c\in\mathcal{C}(C,C^{\prime})$ such that $(a,b,c)$ is a morphism in $\triangleright$.
\item[(RTR4)]Let
\begin{eqnarray*}
A\overset{f}{\longrightarrow}B\overset{g}{\longrightarrow}C\overset{h}{\longrightarrow}\Sigma A,\\
A\overset{\ell}{\longrightarrow}M\overset{m}{\longrightarrow}B^{\prime}\overset{n}{\longrightarrow}\Sigma A,\\
A^{\prime}\overset{\ell^{\prime}}{\longrightarrow}M\overset{m^{\prime}}{\longrightarrow}B\overset{n^{\prime}}{\longrightarrow}\Sigma A^{\prime}
\end{eqnarray*}
be objects in $\triangleright$, satisfying $m^{\prime}\circ\ell=f$.

Then there exist $g^{\prime}\in\mathcal{C}(B^{\prime},C)$ and $h^{\prime}\in\mathcal{C}(C,\Sigma A^{\prime})$ such that
\begin{eqnarray*}
h^{\prime}\circ g=n^{\prime}&,&h\circ g^{\prime}=n,\\
g^{\prime}\circ m=g\circ m^{\prime}&,&(\Sigma\ell)\circ h+(\Sigma\ell^{\prime})\circ h^{\prime}=0,
\end{eqnarray*}
and
\[ A^{\prime}\overset{f^{\prime}}{\longrightarrow}B^{\prime}\overset{g^{\prime}}{\longrightarrow}C^{\prime}\overset{h^{\prime}}{\longrightarrow}\Sigma A^{\prime} \]
is an object in $\triangleright$. Here we put $f^{\prime}=m\circ\ell^{\prime}$.
\[
\xy
(-30,10)*+{A}="2";
(-20,0)*+{M}="4";
(-10,-10)*+{B^{\prime}}="6";
(10,-10)*+{\Sigma A}="8";
(10,10)*+{\Sigma A^{\prime}}="10";
(-10,10)*+{B}="12";
(0,0)*+{C}="14";
(-30,-10)*+{A^{\prime}}="18";
(20,0)*+{\Sigma M}="20";
(-20,14)*+{}="22";
(-20,-14)*+{}="24";
(0,14)*+{}="26";
(0,-14)*+{}="28";
{\ar_{\ell} "2";"4"};
{\ar^{m} "4";"6"};
{\ar_{n} "6";"8"};
{\ar^{f} "2";"12"};
{\ar_{g} "12";"14"};
{\ar^{h} "14";"8"};
{\ar^{\ell^{\prime}} "18";"4"};
{\ar_{m^{\prime}} "4";"12"};
{\ar^{n^{\prime}} "12";"10"};
{\ar^{\Sigma\ell^{\prime}} "10";"20"};
{\ar_{-\Sigma\ell} "8";"20"};
{\ar@{-->}^{g^{\prime}} "6";"14"};
{\ar@{-->}_{f^{\prime}} "18";"6"};
{\ar@{-->}_{h^{\prime}} "14";"10"};
{\ar@{}|\circlearrowright "6";"12"};
{\ar@{}|\circlearrowright "4";"22"};
{\ar@{}|\circlearrowright "4";"24"};
{\ar@{}|\circlearrowright "14";"26"};
{\ar@{}|\circlearrowright "14";"28"};
{\ar@{}|\circlearrowright "8";"10"};
\endxy
\]
\end{itemize}
If $(\Sigma,\triangleright)$ is a right triangulation on $\mathcal{C}$, we call $(\mathcal{C},\Sigma,\triangleright)$ a {\it right triangulated category}.
\end{dfn}

\begin{caution}
Conditions {\rm (RTR4)} is slightly different from that in \cite{BM}.
\end{caution}

\begin{dfn}[left triangulation]
Let $\Omega\colon\mathcal{C}\rightarrow\mathcal{C}$ be an additive endofunctor, and let $\mathcal{LT}(\mathcal{C},\Omega)$ be the category of diagrams of the form
\[ \Omega C\overset{e}{\longrightarrow}A\overset{f}{\longrightarrow}B\overset{g}{\longrightarrow}C. \]
A morphism in $\mathcal{LT}(\mathcal{C},\Omega)$ is defined similarly as in Definition \ref{DefRTR}.
A pair $(\Omega,\triangleleft)$ satisfying conditions {\rm (LTR1)}, {\rm (LTR2)}, {\rm (LTR3)} and {\rm (LTR4)} which are dual to {\rm (RTR1)}, {\rm (RTR2)}, {\rm (RTR3)} and {\rm (RTR4)} respectively, is called a {\it left triangulation} on $\mathcal{C}$, and $(\mathcal{C},\Omega,\triangleleft)$ is called a {\it left triangulated category}.
\end{dfn}

Similarly to the triangulated case, the following are satisfied.
\begin{prop}\label{PropExact}
Let $\mathcal{C}$ be an additive category.
\begin{enumerate}
\item If $(\Sigma, \triangleright)$ is a right triangulation on $\mathcal{C}$, then for any object $A\rightarrow B\rightarrow C\rightarrow\Sigma A$ in $\triangleright$ and for any $E\in\mathcal{C}$, the induced sequence
\[ \mathcal{C}(A,E)\leftarrow\mathcal{C}(B,E)\leftarrow\mathcal{C}(C,E)\leftarrow\mathcal{C}(\Sigma A,E)\leftarrow\mathcal{C}(\Sigma B,E)\leftarrow\cdots \]
is exact.
\item Dually for a left triangulation.
\end{enumerate}
\end{prop}
\begin{proof}
Left to the reader.
\end{proof}

\section{Pseudo-triangulated category}

In this section, we introduce a notion unifying triangulated categories and abelian categories. We make a slight modification of the pretriangulated category in \cite{BR}, for the sake of Example \ref{ExTwoCases}. We call it a \lq pseudo-'triangulated category, to make the reader beware of this modification.
Roughly speaking, a pseudo-triangulated category is an additive category endowed with right and left triangulated triangulations, satisfying some gluing conditions (Definition \ref{DefPreTr}).
\begin{dfn}\label{DefMorphProperty}
Let $(\Sigma,\triangleright)$ be a right triangulation on $\mathcal{C}$, and let $f\colon A\rightarrow B$ be any morphism in $\mathcal{C}$.
\begin{enumerate}
\item $f$ is $\Sigma$-{\it null} if it factors through some object in $\Sigma\mathcal{C}$.
\item $f$ is $\Sigma$-{\it epic} if for any $B^{\prime}\in\mathcal{C}$ and any $b\in\mathcal{C}(B,B^{\prime})$, $b\circ f=0$ implies $b$ is $\Sigma$-null.
\end{enumerate}
\end{dfn}
For a left triangulation $(\Omega,\triangleleft)$, dually we define $\Omega$-{\it null} morphisms and $\Omega$-{\it monic} morphisms.

\begin{rem}
For any morphism $f\in\mathcal{C}(A,B)$, the following are equivalent.
\begin{enumerate}
\item $f$ is $\Sigma$-epic.
\item There exists an object in $\triangleright$
\[ A\overset{f}{\longrightarrow}B\overset{g}{\longrightarrow}C\rightarrow\Sigma A \]
such that $g$ is $\Sigma$-null.
\item For any object in $\triangleright$
\[ A\overset{f}{\longrightarrow}B\overset{g}{\longrightarrow}C\rightarrow\Sigma A, \]
$g$ becomes $\Sigma$-null.
\end{enumerate}

Dually for $\Omega$-monics.
\end{rem}

\begin{dfn}\label{DefPreTr}
A pseudo-triangulation $(\Sigma,\Omega,\triangleright,\triangleleft,\psi)$ on $\mathcal{C}$ is a pair $(\Sigma,\triangleright)$ and $(\Omega,\triangleleft)$ of right and left triangulations, together with an adjoint natural isomorphism
\[ \psi_{A,B}\colon\mathcal{C}(\Omega A,B)\overset{\cong}{\longrightarrow}\mathcal{C}(A,\Sigma B) \quad(A,B\in\mathcal{C}), \]
which satisfies the following gluing conditions {\rm (G1)} and {\rm (G2)}.
\begin{enumerate}
\item[{\rm (G1)}] If $g\in\mathcal{C}(B,C)$ is $\Sigma$-epic, then for any objects
\begin{eqnarray*}
\Omega C\overset{e}{\longrightarrow}A\overset{f}{\longrightarrow}B\overset{g}{\longrightarrow}C&\in&\triangleleft,\\
A\overset{f}{\longrightarrow}B\overset{g^{\prime}}{\longrightarrow}C^{\prime}\overset{h^{\prime}}{\longrightarrow}\Sigma A&\in&\triangleright,
\end{eqnarray*}
there exists an isomorphism $c\in\mathcal{C}(C^{\prime},C)$ such that
\[ c\circ g^{\prime}=g\ \ \text{and}\ \ -\psi(e)\circ c=h^{\prime}. \]
\[
\xy
(-14,6)*+{A}="0";
(0,6)*+{B}="2";
(14,6)*+{C^{\prime}}="4";
(28,6)*+{\Sigma A}="6";
(-28,-6)*+{\Omega C}="-12";
(-14,-6)*+{A}="10";
(0,-6)*+{B}="12";
(14,-6)*+{C}="14";
(24,-2)*+{}="15";
{\ar^{f} "0";"2"};
{\ar^{g^{\prime}} "2";"4"};
{\ar^{h^{\prime}} "4";"6"};
{\ar_{e} "-12";"10"};
{\ar_{f} "10";"12"};
{\ar_{g} "12";"14"};
{\ar@{=} "0";"10"};
{\ar@{=} "2";"12"};
{\ar_{\cong}@{-->}^{{}^{\exists}c} "4";"14"};
{\ar_{-\psi(e)} "14";"6"};
{\ar@{}|\circlearrowright "0";"12"};
{\ar@{}|\circlearrowright "2";"14"};
{\ar@{}|\circlearrowright "4";"15"};
\endxy
\]
Roughly speaking, this means that any $\Sigma$-epic morphism agrees with the \lq cokernel' of its \lq kernel'.
\item[{\rm (G2)}]  Dually, if $f\in\mathcal{C}(A,B)$ is $\Omega$-monic, then for any objects
\begin{eqnarray*}
A\overset{f}{\longrightarrow}B\overset{g}{\longrightarrow}C\overset{h}{\longrightarrow}\Sigma A&\in&\triangleright,\\
\Omega C\overset{e^{\prime}}{\longrightarrow}A^{\prime}\overset{f^{\prime}}{\longrightarrow}B\overset{g}{\longrightarrow}C&\in&\triangleleft,
\end{eqnarray*}
there exists an isomorphism $a\in\mathcal{C}(A,A^{\prime})$ such that
\[ f^{\prime}\circ a=f\ \ \text{and}\ \ -a\circ\psi^{-1}(h)=e^{\prime}. \]
\[
\xy
(-14,6)*+{A}="0";
(0,6)*+{B}="2";
(14,6)*+{C}="4";
(28,6)*+{\Sigma A}="6";
(-28,-6)*+{\Omega C}="-12";
(-14,-6)*+{A^{\prime}}="10";
(0,-6)*+{B}="12";
(14,-6)*+{C}="14";
(-24,2)*+{}="-1";
{\ar^{f} "0";"2"};
{\ar^{g} "2";"4"};
{\ar^{h} "4";"6"};
{\ar_{e^{\prime}} "-12";"10"};
{\ar_{f^{\prime}} "10";"12"};
{\ar_{g} "12";"14"};
{\ar_{{}^{\exists}a}^{\cong} "0";"10"};
{\ar@{=} "2";"12"};
{\ar@{=} "4";"14"};
{\ar^{-\psi^{-1}(h)} "-12";"0"};
{\ar@{}|\circlearrowright "0";"12"};
{\ar@{}|\circlearrowright "2";"14"};
{\ar@{}|\circlearrowright "-1";"10"};
\endxy
\]
\end{enumerate}
If we are given a pseudo-triangulation $(\Sigma,\Omega,\triangleright,\triangleleft,\psi)$ on $\mathcal{C}$, then we call the 6-tuple $(\mathcal{C},\Sigma,\Omega,\triangleright,\triangleleft,\psi)$ a {\it pseudo-triangulated category}.
We often represent a pseudo-triangulated category simply by $\mathcal{C}$.
\end{dfn}

\begin{ex}\label{ExTwoCases0}
Let $(\mathcal{C},\Sigma,\Omega,\triangleright,\triangleleft,\psi)$ be a pseudo-triangulated category. 
\begin{enumerate}
\item $\mathcal{C}$ is an abelian category if and only if $\Sigma=\Omega=0$.
\item $\mathcal{C}$ is a triangulated category if and only if $\Sigma$ is the quasi-inverse of $\Omega$ and $\psi$ is the one induced from the isomorphism $\Sigma\circ\Omega\cong\mathrm{Id}_{\mathcal{C}}$.
\end{enumerate}
\end{ex}
\begin{proof}
{\rm (1)} We only show that $\Sigma=\Omega=0$ implies the abelianess of $\mathcal{C}$. The converse is confirmed by a routine work.
Since $\Sigma=0$, Proposition \ref{PropExact} means $g=\mathrm{cok}(f)$ holds for any object
\[ A\overset{f}{\longrightarrow}B\overset{g}{\longrightarrow}C\overset{h}{\longrightarrow}\Sigma A \]
in $\triangleright$.

Thus {\rm (RTR1)} implies the existence of a cokernel for each morphism.
Dually for the existence of $\mathrm{ker}(f)$.
Moreover, in this case $f$ is $\Sigma$-null if and only if $f=0$, and $f$ is $\Sigma$-epic if and only if it is epimorphic.
Thus {\rm (G1)} means that any epimorphism $g$ agrees with $\mathrm{cok}(\mathrm{ker}(g))$. Dually for monomorphisms.

{\rm (2)} In this case, any morphism is at the same time $\Sigma$-null and $\Sigma$-epic, and $\Omega$-null and $\Omega$-monic. Moreover, $\triangleright$ and $\triangleleft$ agree. We only show $\triangleleft\subseteq\triangleright$.

By {\rm (LTR2)}, for any object
\begin{equation}
\Omega C\overset{e}{\longrightarrow}A\overset{f}{\longrightarrow}B\overset{g}{\longrightarrow}C
\label{TempEq}
\end{equation}
in $\triangleleft$, the shifted one
\[ \Omega B\overset{-\Omega g}{\longrightarrow}\Omega C\overset{e}{\longrightarrow}A\overset{f}{\longrightarrow}B \]
is also in $\triangleleft$. By {\rm (G1)}, we obtain an object in $\triangleright$
\[ \Omega C\overset{e}{\longrightarrow}A\overset{f}{\longrightarrow}B\overset{\psi(\Omega g)}{\longrightarrow}\Sigma\Omega C, \]
which is isomorphic to $(\ref{TempEq})$.
\end{proof}

\section{Extensions}
In this section, $\mathcal{C}$ is a pseudo-triangulated category with pseudo-triangulation $(\Sigma,\Omega,\triangleright,\triangleleft,\psi)$. We define the notion of an extension which generalizes a short exact sequence in an abelian category, and a distinguished triangle in a triangulated category.

\begin{dfn}\label{DefExtension}
A sequence in $\mathcal{C}$
\[ \Omega C\overset{e}{\longrightarrow}A\overset{f}{\longrightarrow} B\overset{g}{\longrightarrow}C\overset{h}{\longrightarrow}\Sigma A \]
is called an {\it extension} if it satisfies
\begin{eqnarray*}
&(A\overset{f}{\longrightarrow} B\overset{g}{\longrightarrow}C\overset{h}{\longrightarrow}\Sigma A)\in\triangleright,&\\
&(\Omega C\overset{e}{\longrightarrow}A\overset{f}{\longrightarrow} B\overset{g}{\longrightarrow}C)\in\triangleleft,&\\
&h=-\psi_{C,A}(e).&
\end{eqnarray*}
Since $e$ and $h$ determines each other, we sometimes omit one of them.

A {\it morphism of extensions} from
\[ \Omega C\overset{e}{\longrightarrow}A\overset{f}{\longrightarrow} B\overset{g}{\longrightarrow}C\overset{h}{\longrightarrow}\Sigma A \]
to
\[ \Omega C^{\prime}\overset{e^{\prime}}{\longrightarrow}A^{\prime}\overset{f^{\prime}}{\longrightarrow} B^{\prime}\overset{g^{\prime}}{\longrightarrow}C^{\prime}\overset{h^{\prime}}{\longrightarrow}\Sigma A^{\prime} \]
is a triplet $(a,b,c)$ of $a\in\mathcal{C}(A,A^{\prime})$, $b\in\mathcal{C}(B,B^{\prime})$ and $c\in\mathcal{C}(C,C^{\prime})$ satisfying
\[ b\circ f=f^{\prime}\circ a,\quad c\circ g=g^{\prime}\circ b,\quad (\Sigma a)\circ h=h^{\prime}\circ c. \]
Remark that $(\Sigma a)\circ h=h^{\prime}\circ c$ is equivalent to $a\circ e=e^{\prime}\circ(\Omega c)$. Thus, a morphism of extensions is essentially the same as a morphism in $\triangleright$ or $\triangleleft$.
\[
\xy
(-28,6)*+{\Omega C}="-2";
(-14,6)*+{A}="0";
(0,6)*+{B}="2";
(14,6)*+{C}="4";
(28,6)*+{\Sigma A}="6";
(-28,-6)*+{\Omega C^{\prime}}="-12";
(-14,-6)*+{A^{\prime}}="10";
(0,-6)*+{B^{\prime}}="12";
(14,-6)*+{C^{\prime}}="14";
(28,-6)*+{\Sigma A^{\prime}}="16";
{\ar^{e} "-2";"0"};
{\ar^{f} "0";"2"};
{\ar^{g} "2";"4"};
{\ar^{h} "4";"6"};
{\ar_{e^{\prime}} "-12";"10"};
{\ar_{f^{\prime}} "10";"12"};
{\ar_{g^{\prime}} "12";"14"};
{\ar_{h^{\prime}} "14";"16"};
{\ar_{\Omega c} "-2";"-12"};
{\ar^{a} "0";"10"};
{\ar^{b} "2";"12"};
{\ar^{c} "4";"14"};
{\ar^{\Sigma a} "6";"16"};
{\ar@{}|\circlearrowright "-2";"10"};
{\ar@{}|\circlearrowright "0";"12"};
{\ar@{}|\circlearrowright "2";"14"};
{\ar@{}|\circlearrowright "4";"16"};
\endxy
\]
\end{dfn}

\begin{rem}\label{RemExtension}
Consider a diagram in $\mathcal{C}$
\begin{equation}
\Omega C\overset{e}{\longrightarrow}A\overset{f}{\longrightarrow} B\overset{g}{\longrightarrow}C\overset{h}{\longrightarrow}\Sigma A
\label{ExtensionDiag}
\end{equation}
satisfying $h=-\psi(e)$. By {\rm (G1)} and {\rm (G2)} (and {\rm (RTR1)} and {\rm (LTR1)}), the following are equivalent.
\begin{enumerate}
\item $\Omega C\overset{e}{\longrightarrow}A\overset{f}{\longrightarrow} B\overset{g}{\longrightarrow}C$ belongs to $\triangleleft$ and $g$ is $\Sigma$-epic.
\item $A\overset{f}{\longrightarrow} B\overset{g}{\longrightarrow}C\overset{h}{\longrightarrow}\Sigma A$ belongs to $\triangleright$ and $f$ is $\Omega$-monic.
\item $(\ref{ExtensionDiag})$ is an extension.
\end{enumerate}
\end{rem}

\begin{cor}\label{CorExtension}
$\ \ $
\begin{enumerate}
\item $g\in\mathcal{C}(B,C)$ is $\Sigma$-epic if and only if there exists an extension $(\ref{ExtensionDiag})$, if and only if there exists an object $A\rightarrow B\overset{g}{\longrightarrow}C\rightarrow\Sigma A$ in $\triangleright$.
\item $f\in\mathcal{C}(A,B)$ is $\Omega$-monic if and only if there exists an extension $(\ref{ExtensionDiag})$, if and only if there exists an object $\Omega C\rightarrow A\overset{f}{\longrightarrow}B\rightarrow C$ in $\triangleleft$.
\end{enumerate}
\end{cor}
\begin{proof}
We show only {\rm (1)}. If there exists an object $A\rightarrow B\overset{g}{\longrightarrow}C\rightarrow\Sigma A$ in $\triangleright$, then by {\rm (RTR2)}, we have an object in $\triangleright$
\[ B\overset{g}{\longrightarrow}C\rightarrow\Sigma A\rightarrow \Sigma B. \]
Obviously this implies $g$ is $\Sigma$-epic.

Conversely if $g$ is $\Sigma$-epic, then by {\rm (LTR1)} and Remark \ref{RemExtension}, we obtain an extension $(\ref{ExtensionDiag})$.
\end{proof}

\begin{lem}\label{LemSPL}
Let $f\in\mathcal{C}(A,B)$, $m\in\mathcal{C}(A,M)$ and $e\in\mathcal{C}(M,B)$ be morphisms satisfying $e\circ m=f$.
\begin{enumerate}
\item If $f$ is $\Sigma$-epic, then so is $e$.
\item If $f$ is $\Omega$-monic, then so is $m$.
\end{enumerate}
\end{lem}
\begin{proof}
{\rm (1)} By {\rm (RTR1)} and {\rm (RTR3)}, there exists a morphism in $\triangleright$
\[
\xy
(-14,6)*+{A}="0";
(0,6)*+{B}="2";
(14,6)*+{C}="4";
(28,6)*+{\Sigma A}="6";
(-14,-6)*+{M}="10";
(0,-6)*+{B}="12";
(14,-6)*+{D}="14";
(28,-6)*+{\Sigma M.}="16";
{\ar^{f} "0";"2"};
{\ar^{g} "2";"4"};
{\ar^{h} "4";"6"};
{\ar_{e} "10";"12"};
{\ar_{g^{\prime}} "12";"14"};
{\ar_{h^{\prime}} "14";"16"};
{\ar_{m} "0";"10"};
{\ar@{=} "2";"12"};
{\ar^{} "4";"14"};
{\ar^{\Sigma m} "6";"16"};
{\ar@{}|\circlearrowright "0";"12"};
{\ar@{}|\circlearrowright "2";"14"};
{\ar@{}|\circlearrowright "4";"16"};
\endxy
\]
Then since $g$ is $\Sigma$-null, so is $g^{\prime}$.
{\rm (2)} is shown dually.
\end{proof}

\begin{ex}\label{ExTwoCases}
The notion of an extension becomes as follows in the two cases of Example \ref{ExTwoCases0}.
\begin{enumerate}
\item If $\Sigma=\Omega=0$ and $\mathcal{C}$ is abelian, then an extension is nothing other than a short exact sequence.
\item If $\mathcal{C}$ is a triangulated category as in Example \ref{ExTwoCases0}, then an extension is nothing other than a distinguished triangle.
\end{enumerate}
\end{ex}

\begin{prop}\label{PropSPL}
For any $A,B\in\mathcal{C}$,
\[ \Omega B\overset{0}{\longrightarrow}A\overset{i_A}{\longrightarrow}A\oplus B\overset{p_B}{\longrightarrow}B\overset{0}{\longrightarrow}\Sigma A \]
is an extension, where $i_A$ and $p_B$ are the injection and the projection, respectively.
\end{prop}
\begin{proof}
Let $p_A\colon A\oplus B\rightarrow A$ be the projection, and $i_B\colon B\rightarrow A\oplus B$ be the inclusion. Since $\mathrm{id}_B$ is $\Sigma$-epic by {\rm (RTR1)}, so is $p_B$ by Lemma \ref{LemSPL}. Thus by Corollary \ref{CorExtension}, there is an extension
\[ \Omega B\overset{u}{\longrightarrow}{}^{\exists}C\overset{v}{\longrightarrow}A\oplus B\overset{p_B}{\longrightarrow}B\overset{w}{\longrightarrow}\Sigma C \]
with some morphisms $u,v,w$.
Since $p_B$ is the projection and $w\circ p_B=0$ by Proposition \ref{PropExact}, we have $w=0$, and thus $u=0$. By $p_B\circ i_A=0$, there exists $r\in\mathcal{C}(A,C)$ such that $v\circ r=i_A$.
\[
\xy
(-22,-6)*+{\Omega B}="0";
(-8,-6)*+{C}="2";
(6,-6)*+{A\oplus B}="4";
(22,-6)*+{B}="6";
(6,6)*+{A}="14";
(-2,2)*+{}="15";
{\ar_{u=0} "0";"2"};
{\ar_{v} "2";"4"};
{\ar_{p_B} "4";"6"};
{\ar^{i_A} "14";"4"};
{\ar@{-->}_{{}^{\exists}r} "14";"2"};
{\ar@{}|\circlearrowright "15";"4"};
\endxy
\]
Then we have
\begin{eqnarray*}
v\circ(\mathrm{id}_C-r\circ(p_A\circ v))&=&v-v\circ r\circ p_A\circ v\\
&=&(\mathrm{id_C}-i_A\circ p_A)\circ v\\
&=&(i_B\circ p_B)\circ v=0.
\end{eqnarray*}
Thus $\mathrm{id}_C-r\circ p_A\circ v$ factors through $u=0$, which means
\[ r\circ(p_A\circ v)=\mathrm{id}_C. \]
Since $(p_A\circ v)\circ r=p_A\circ i_A=\mathrm{id}_A$, this means $r$ is an isomorphism.
\end{proof}

\begin{prop}\label{PropOCT}
Let
\begin{eqnarray*}
\Omega C\overset{e}{\longrightarrow}A\overset{f}{\longrightarrow}B\overset{g}{\longrightarrow}C\overset{h}{\longrightarrow}\Sigma A,\\
\Omega B^{\prime}\overset{k}{\longrightarrow}A\overset{\ell}{\longrightarrow}M\overset{m}{\longrightarrow}B^{\prime}\overset{n}{\longrightarrow}\Sigma A,\\
\Omega B\overset{k^{\prime}}{\longrightarrow}A^{\prime}\overset{\ell^{\prime}}{\longrightarrow}M\overset{m^{\prime}}{\longrightarrow}B\overset{n^{\prime}}{\longrightarrow}\Sigma A^{\prime},
\end{eqnarray*}
be extensions, satisfying $m^{\prime}\circ \ell=f$.
Then there exist $g^{\prime}\in\mathcal{C}(B^{\prime},C)$ and $h^{\prime}\in\mathcal{C}(C,\Sigma A^{\prime})$ such that
\begin{eqnarray*}
h^{\prime}\circ g=n^{\prime}&,&h\circ g^{\prime}=n,\\
g^{\prime}\circ m=g\circ m^{\prime}&,&(\Sigma\ell)\circ h+(\Sigma\ell^{\prime})\circ h^{\prime}=0,
\end{eqnarray*}
and
\[ \Omega C\rightarrow A^{\prime}\overset{f^{\prime}}{\longrightarrow}B^{\prime}\overset{g^{\prime}}{\longrightarrow}C\overset{h^{\prime}}{\longrightarrow}\Sigma A^{\prime} \]
is an extension. Here we put $f^{\prime}=m\circ\ell^{\prime}$.
Remark if we put $e^{\prime}=-\psi^{-1}(h^{\prime})$, then $(\Sigma\ell)\circ h+(\Sigma\ell^{\prime})\circ h^{\prime}=0$ is equivalent to $\ell^{\prime}\circ e^{\prime}+\ell\circ e=0$.
\[
\xy
(-32,0)*+{\Omega B^{\prime}}="0";
(-16,0)*+{A}="2";
(0,0)*+{M}="4";
(16,0)*+{B^{\prime}}="6";
(32,0)*+{\Sigma A}="8";
(0,24)*+{\Sigma A^{\prime}}="10";
(32,24)*+{\Sigma M}="30";
(0,12)*+{B}="12";
(10,20)*+{}="13";
(16,12)*+{C}="14";
(13,9)*+{}="15";
(-16,-12)*+{\Omega C}="16";
(0,-12)*+{A^{\prime}}="18";
(0,-24)*+{\Omega B}="20";
(-10,8)*+{}="22";
(10,-8)*+{}="24";
(26,8)*+{}="26";
{\ar^{k} "0";"2"};
{\ar_{\ell} "2";"4"};
{\ar_{m} "4";"6"};
{\ar_{n} "6";"8"};
{\ar^{e} "16";"2"};
{\ar^{f} "2";"12"};
{\ar_{g} "12";"14"};
{\ar^{h} "14";"8"};
{\ar^{k^{\prime}} "20";"18"};
{\ar^{\ell^{\prime}} "18";"4"};
{\ar_{m^{\prime}} "4";"12"};
{\ar^{n^{\prime}} "12";"10"};
{\ar@{-->}^{g^{\prime}} "6";"14"};
{\ar@{-->}_{e^{\prime}} "16";"18"};
{\ar@{-->}_{f^{\prime}} "18";"6"};
{\ar@{-->}_{h^{\prime}} "14";"10"};
{\ar^{-\Sigma\ell} "8";"30"};
{\ar_{\Sigma\ell^{\prime}} "10";"30"};
{\ar@{}|\circlearrowright "6";"12"};
{\ar@{}|\circlearrowright "4";"22"};
{\ar@{}|\circlearrowright "4";"24"};
{\ar@{}|\circlearrowright "6";"26"};
{\ar@{}|\circlearrowright "12";"13"};
{\ar@{}|\circlearrowright "15";"30"};
\endxy
\]
Dual statement also holds.
\end{prop}
\begin{proof}
By {\rm (RTR4)}, there exist $g^{\prime}\in\mathcal{C}(B^{\prime},C)$ and $h^{\prime}\in\mathcal{C}(C,\Sigma A^{\prime})$ such that
\begin{eqnarray*}
h^{\prime}\circ g=n^{\prime}&,&h\circ g^{\prime}=n,\\
g^{\prime}\circ m=g\circ m^{\prime}&,&(\Sigma\ell)\circ h+(\Sigma\ell^{\prime})\circ h^{\prime}=0,
\end{eqnarray*}
and
\[ A^{\prime}\overset{f^{\prime}}{\longrightarrow}B^{\prime}\overset{g^{\prime}}{\longrightarrow}C\overset{h^{\prime}}{\longrightarrow}\Sigma A^{\prime} \]
is an object in $\triangleright$.
Thus by Remark \ref{RemExtension}, it suffices to show $f^{\prime}$ is $\Omega$-monic. This follows from {\rm (LTR4)}. In fact, applying {\rm (LTR4)} to objects in $\triangleleft$
\begin{eqnarray*}
\Omega B\overset{-\Omega g}{\longrightarrow}\Omega C\overset{e}{\longrightarrow}A\overset{f}{\longrightarrow}B,\\
\Omega B^{\prime}\overset{k}{\longrightarrow}A\overset{\ell}{\longrightarrow}M\overset{m}{\longrightarrow}B^{\prime},\\
\Omega B\overset{k^{\prime}}{\longrightarrow}A^{\prime}\overset{\ell^{\prime}}{\longrightarrow}M\overset{m^{\prime}}{\longrightarrow}B,
\end{eqnarray*}
we obtain an object in $\triangleleft$
\[ \Omega B^{\prime}\rightarrow\Omega C\rightarrow A^{\prime}\overset{f^{\prime}}{\longrightarrow}B^{\prime}, \]
which means $f^{\prime}$ is $\Omega$-monic.
\end{proof}

\section{Frobenius condition}
\label{SecFrobCond}
In this section, we define an extension-closed subcategory $\mathcal{Z}$ of $\mathcal{C}$, and the Frobenius condition on it. This condition generalizes simultaneously the usual Frobenius condition for an exact category, and the the existence of a subcategory $\mathcal{D}$ such that $(\mathcal{Z},\mathcal{Z})$ is a $\mathcal{D}$-mutation pair in the case of a triangulated category.

\begin{dfn}
A subcategory $\mathcal{Z}\subseteq\mathcal{C}$ is said to be {\it extension-closed} if it satisfies the following.
\begin{itemize}
\item[$(\ast)$] For any extension in $\mathcal{C}$
\[ \Omega Z\overset{e}{\longrightarrow}X\overset{f}{\longrightarrow}Y\overset{g}{\longrightarrow}Z\overset{h}{\longrightarrow}\Sigma X, \]
$X,Z\in\mathcal{Z}$ implies $Y\in\mathcal{Z}$.
\end{itemize}
\end{dfn}

In the following, we fix an extension-closed subcategory $\mathcal{Z}\subseteq\mathcal{C}$.

\begin{rem}
When $\mathcal{C}$ is an abelian category as in Example \ref{ExTwoCases}, then $\mathcal{Z}$ is an exact category.
\end{rem}

\begin{dfn}Let $\mathcal{Z}\subseteq\mathcal{C}$ be an extension-closed subcategory as above.
\begin{enumerate}
\item A {\it conflation} is an extension in $\mathcal{C}$
\begin{equation}
\Omega Z\overset{e}{\longrightarrow}X\overset{f}{\longrightarrow}Y\overset{g}{\longrightarrow}Z\overset{h}{\longrightarrow}\Sigma X,
\label{conflation}
\end{equation}
satisfying $X,Y,Z\in\mathcal{Z}$. A {\it morphism of conflations} is a morphism of the extensions.
\item A morphism $f\colon X\rightarrow Y$ in $\mathcal{Z}$ is an {\it inflation} if there exists a conflation $(\ref{conflation})$.
\item A morphism $g\colon Y\rightarrow Z$ in $\mathcal{Z}$ is a {\it deflation} if there exists a conflation $(\ref{conflation})$.
\end{enumerate}
\end{dfn}

In the following, we fix an extension-closed subcategory $\mathcal{Z}\subseteq\mathcal{C}$. For a full additive replete subcategory $\mathcal{D}\subseteq\mathcal{Z}$, we consider the following condition {\rm (DS)}.
\begin{condition}
$\ \ $
\begin{itemize}
\item[(DS)]
$\mathcal{D}$ is {\it closed under finite direct summands in} $\mathcal{Z}$, namely, for any $Z_1,Z_2\in\mathcal{Z}$ and $D\in\mathcal{D}$, $D\cong Z_1\oplus Z_2$ implies $Z_1,Z_2\in\mathcal{Z}$.
\end{itemize}
\end{condition}

\begin{dfn}
Let $\mathcal{D}\subseteq\mathcal{Z}$ be a full additive replete subcategory satisfying {\rm (DS)}.
\begin{enumerate}
\item An object $I$ in $\mathcal{D}$ is {\it injective} if
\[ \mathcal{Z}(Y,I)\overset{-\circ f}{\longrightarrow}\mathcal{Z}(X,I)\rightarrow 0 \]
is exact for any inflation $f\colon X\rightarrow Y$. We denote the full subcategory of injective objects by $\mathcal{I}_{\mathcal{D}}\subseteq\mathcal{D}$. In particular $\mathcal{I}_{\mathcal{Z}}$ is denoted by $\mathcal{I}$.
\item An object $P$ in $\mathcal{D}$ is {\it projective} if
\[ \mathcal{Z}(P,Y)\overset{g\circ-}{\longrightarrow}\mathcal{Z}(P,Z)\rightarrow 0 \]
is exact for any deflation $g\colon Y\rightarrow Z$. We denote the full subcategory of projective objects by $\mathcal{P}_{\mathcal{D}}\subseteq\mathcal{D}$. In particular $\mathcal{P}_{\mathcal{Z}}$ is denoted by $\mathcal{P}$.
\end{enumerate}
\end{dfn}

\begin{ex}\label{ExExtSub}$\ \ $
\begin{enumerate}
\item If $\mathcal{Z}\subseteq\mathcal{C}$ is an exact category where $\mathcal{C}$ is an abelian category as in Example \ref{ExTwoCases}, then $\mathcal{I}$ is equal to the full subcategory of injective objects, and $\mathcal{P}$ is equal to the full subcategory of projective objects.

\item If $\mathcal{C}$ is a triangulated category, and if $\mathcal{D}$ satisfies $\mathcal{C}(\Omega\mathcal{Z},\mathcal{D})=\mathcal{C}(\mathcal{D},\Sigma\mathcal{Z})=0$, then we have $\mathcal{I}_{\mathcal{D}}=\mathcal{P}_{\mathcal{D}}=\mathcal{D}$.
\end{enumerate}
\end{ex}

\begin{caution}
The definitions of injective and projective objects are different from those in \cite{B}.
\end{caution}

\begin{rem}\label{RemIPThick}
$\ \ $
\begin{enumerate}
\item $\mathcal{I}_{\mathcal{D}}$ and $\mathcal{P}_{\mathcal{D}}$ are full additive replete subcategories, which are closed under finite direct summands in $\mathcal{Z}$.
\item $\mathcal{I}_{\mathcal{D}}=\mathcal{I}\cap\mathcal{D}$.
\item $\mathcal{P}_{\mathcal{D}}=\mathcal{P}\cap\mathcal{D}$.
\end{enumerate}
\end{rem}
\begin{proof}
Left to the reader.
\end{proof}

\begin{dfn}\label{DefFrobCond}
Let $(\mathcal{C},\mathcal{Z},\mathcal{D})$ be a triplet as above.
\begin{enumerate}
\item $(\mathcal{C},\mathcal{Z},\mathcal{D})$ {\it has enough injectives} if for any $X\in\mathcal{Z}$, there exists an inflation $\alpha\colon X\rightarrow I$ such that $I\in\mathcal{I}_{\mathcal{D}}$. When $\mathcal{D}=\mathcal{Z}$, we simply say \lq\lq$\mathcal{Z}$ has enough injectives".
\item $(\mathcal{C},\mathcal{Z},\mathcal{D})$ {\it has enough projectives} if for any $Z\in\mathcal{Z}$, there exists a deflation $\beta\colon P\rightarrow Z$ such that $P\in\mathcal{P}_{\mathcal{D}}$. When $\mathcal{D}=\mathcal{Z}$, we simply say \lq\lq$\mathcal{Z}$ has enough projectives".
\item $(\mathcal{C},\mathcal{Z},\mathcal{D})$ is {\it Frobenius} if it has enough injectives and projectives, and moreover $\mathcal{I}_{\mathcal{D}}=\mathcal{P}_{\mathcal{D}}$. 
When $\mathcal{D}=\mathcal{Z}$, we simply say \lq\lq$\mathcal{Z}$ is Frobenius".
\end{enumerate}
\end{dfn}

\begin{ex}\label{ExFrobCond}
$\ \ $
\begin{enumerate}
\item If $\mathcal{Z}\subseteq\mathcal{C}$ is an exact category as in Example \ref{ExExtSub}, then $\mathcal{Z}$ is Frobenius if and only if $\mathcal{Z}$ is Frobenius as an exact category. In this case the stable category $\mathcal{Z}/\mathcal{I}$ is triangulated \cite{H}.
\item If $\mathcal{C}$ is a triangulated category and if $(\mathcal{Z},\mathcal{Z})$ is a $\mathcal{D}$-mutation pair in $\mathcal{C}$ $($in the definition in \cite{IY}$)$, then $(\mathcal{C},\mathcal{Z},\mathcal{D})$ is Frobenius. In this case $\mathcal{Z}/\mathcal{I}_{\mathcal{D}}=\mathcal{Z}/\mathcal{D}$ becomes a triangulated category by Theorem 4.2 in \cite{IY}.
\end{enumerate}
\end{ex}

\begin{center}
\begin{tabular}
[c]{|c|c|c|}\hline
& Happel's construction \cite{H} & Iyama and Yoshino's construction \cite{IY} \\\hline
$\mathcal{C}$ & abelian category & triangulated category \\\hline
$\mathcal{Z}$ & exact subcategory & extension-closed subcategory \\\hline
$\mathcal{D}$ & $\mathcal{Z}=\mathcal{D}$ & $(\mathcal{Z},\mathcal{Z})$ : $\mathcal{D}$-mutation pair \\\hline
$\mathcal{I}_{\mathcal{D}}$ & injective objects & $\mathcal{I}_{\mathcal{D}}$=$\mathcal{D}$ \\\hline
$\mathcal{P}_{\mathcal{D}}$ & projective objects & $\mathcal{P}_{\mathcal{D}}$=$\mathcal{D}$ \\\hline
\end{tabular}
\end{center}
In section \ref{SecMainThm}, in a pseudo-triangulated category $\mathcal{C}$ satisfying Condition \ref{CondPTR}, we show $\mathcal{Z}/\mathcal{I}_{\mathcal{D}}$ becomes a triangulated category for any Frobenius triplet $(\mathcal{C},\mathcal{Z},\mathcal{D})$ (Theorem \ref{MainThm}), which we call the {\it stable category} associated to $(\mathcal{C},\mathcal{Z},\mathcal{D})$.
In particular, if $\mathcal{Z}$ is Frobenius, then $\mathcal{Z}/\mathcal{I}$ becomes a triangulated category. We call $\mathcal{Z}/\mathcal{I}$ the stable category associated to $\mathcal{Z}$.

Although we have defined the Frobenius condition on a triplet $(\mathcal{C},\mathcal{Z},\mathcal{D})$, it is essentially the same as the Frobenius condition on $\mathcal{Z}$ as follows (Corollary \ref{CorMax}).

\begin{prop}\label{Prop3to2}
Let $\mathcal{D}\subseteq\mathcal{D}^{\prime}\subseteq\mathcal{Z}$ be full additive replete subcategories satisfying {\rm (DS)}. If $(\mathcal{C},\mathcal{Z},\mathcal{D})$ is Frobenius, so is $(\mathcal{C},\mathcal{Z},\mathcal{D}^{\prime})$. Moreover, we have $\mathcal{I}_{\mathcal{D}^{\prime}}=\mathcal{I}_{\mathcal{D}}$.
\end{prop}
\begin{proof}
This immediately follows from the lemma below.
\end{proof}

\begin{lem}\label{LemNotDepend}
Let $\mathcal{D}\subseteq\mathcal{D}^{\prime}\subseteq\mathcal{Z}$ be as in Proposition \ref{Prop3to2}. If $(\mathcal{C},\mathcal{Z},\mathcal{D})$ has enough injectives, then we have $\mathcal{I}_{\mathcal{D}^{\prime}}=\mathcal{I}_{\mathcal{D}}$. Similarly for projectives.
\end{lem}
\begin{proof}
Remark that $\mathcal{I}_{\mathcal{D}}=\mathcal{I}_{\mathcal{D^{\prime}}}\cap\mathcal{D}$. Thus it suffices to show $\mathcal{I}_{\mathcal{D}^{\prime}}\subseteq\mathcal{D}$.

Since $(\mathcal{C},\mathcal{Z},\mathcal{D})$ has enough injectives, for any $I^{\prime}\in\mathcal{I}_{\mathcal{D}^{\prime}}$, there exists a conflation
\[ \Omega Z\overset{e}{\longrightarrow} I^{\prime}\overset{f}{\longrightarrow}I\overset{g}{\longrightarrow}Z\overset{h}{\longrightarrow}\Sigma I^{\prime}, \]
where $Z\in\mathcal{Z}$ and $I\in\mathcal{I}_{\mathcal{D}}$. Since $I^{\prime}\in\mathcal{I}_{\mathcal{D}^{\prime}}$, there exists $p\in\mathcal{Z}(I,I^{\prime})$ such that $p\circ f=\mathrm{id}_{I^{\prime}}$. By $f\circ e=0$, we have $e=p\circ f\circ e=0$, and thus $h=0$. By $(\mathrm{id}_I-f\circ p)\circ f=0$, there exists $s\in\mathcal{Z}(Z,I)$ such that $s\circ g=\mathrm{id}_I-f\circ p$. Since $(\mathrm{id}_Z-g\circ s)\circ g=0$, $\mathrm{id}_Z-g\circ s$ factors through $h=0$, namely, we have $\mathrm{id}_Z=g\circ s$. Thus we obtain $I=I^{\prime}\oplus Z$. Since $\mathcal{D}$ is closed under finite direct summands in $\mathcal{Z}$, it follows $I^{\prime}\in\mathcal{D}$.
\end{proof}

Thus if $(\mathcal{C},\mathcal{Z},\mathcal{D})$ is a Frobenius triplet, then $\mathcal{Z}$ is Frobenius, and satisfies $\mathcal{I}=\mathcal{I}_{\mathcal{D}}$. In particular, their stable categories are equivalent.

\begin{cor}\label{CorMax}
For any extension-closed subcategory $\mathcal{Z}\subseteq\mathcal{C}$, the following are equivalent.
\begin{enumerate}
\item $\mathcal{Z}$ is Frobenius.
\item There exists a full additive replete subcategory $\mathcal{D}\subseteq\mathcal{Z}$ satisfying {\rm (DS)} such that $(\mathcal{C},\mathcal{Z},\mathcal{D})$ is Frobenius.
\end{enumerate}
\end{cor}
Moreover, there exists the minimum one.
\begin{cor}\label{CorCor}
If $\mathcal{Z}$ is Frobenius, there exists the minimum $\mathcal{D}$, which makes $(\mathcal{C},\mathcal{Z},\mathcal{D})$ Frobenius.
\end{cor}
\begin{proof}
We show $\mathcal{I}$ satisfies the desired conditions. By Remark \ref{RemIPThick}, $\mathcal{I}\subseteq\mathcal{Z}$ is a full additive replete subcategory satisfying {\rm (DS)}. If $\mathcal{Z}$ is Frobenius, it immediately follows that
\[ {\mathcal{I}_{\mathcal{I}}}=\mathcal{I}=\mathcal{P}=\mathcal{P}_{\mathcal{I}}, \]
and $(\mathcal{C},\mathcal{Z},\mathcal{I})$ becomes Frobenius.
Obviously $\mathcal{I}$ is the minimum one, since any Frobenius triplet $(\mathcal{C},\mathcal{Z},\mathcal{D})$ satisfies $\mathcal{I}=\mathcal{I}_{\mathcal{D}}\subseteq\mathcal{D}$.
\end{proof}

When $\mathcal{C}$ is a triangulated category and if $\mathcal{D}\subseteq\mathcal{Z}$ is a full additive replete subcategory satisfying {\rm (DS)} and
\[ \mathcal{C}(\Omega\mathcal{Z},\mathcal{D})=\mathcal{C}(\mathcal{D},\Sigma\mathcal{Z})=0, \]
then $(\mathcal{C},\mathcal{Z},\mathcal{D})$ is Frobenius if and only if $(\mathcal{Z},\mathcal{Z})$ is a $\mathcal{D}$-mutation pair. (We also remark that if there exists one such $\mathcal{D}$, then it is unique and must agree with the full subcategory of $\mathcal{Z}$ consisting of those $D\in\mathcal{Z}$ satisfying $\mathcal{C}(\Omega\mathcal{Z},D)=\mathcal{C}(D,\Sigma\mathcal{Z})=0$.)

Namely, we have the following.
\begin{claim}
Let $\mathcal{D}\subseteq\mathcal{Z}$ be a full additive replete subcategory satisfying {\rm (DS)}. The following are equivalent.
\begin{enumerate}
\item $(\mathcal{C},\mathcal{Z},\mathcal{D})$ is Frobenius, and $\mathcal{C}(\Omega\mathcal{Z},\mathcal{D})=\mathcal{C}(\mathcal{D},\Sigma\mathcal{Z})=0$.
\item $(\mathcal{Z},\mathcal{Z})$ is a $\mathcal{D}$-mutation pair.
\end{enumerate}
\end{claim}

Regarding  Corollary \ref{CorMax} and Corollary \ref{CorCor}, we obtain the following.
\begin{cor}\label{CorFrobenius}
For any $\mathcal{Z}$, the following are equivalent.
\begin{enumerate}
\item $\mathcal{Z}$ is Frobenius, and $\mathcal{C}(\Omega\mathcal{Z},\mathcal{I})=\mathcal{C}(\mathcal{I},\Sigma\mathcal{Z})=0$.
\item $(\mathcal{Z},\mathcal{Z})$ is an $\mathcal{I}$-mutation pair.
\end{enumerate}
\end{cor}

\section{Triangulation on the stable category}
\label{SecMainThm}
In this section, as a main theorem, we show give a triangulation on the stable category associated to an extension-closed subcategory of a pseudo-triangulated category satisfying the following condition.
Remark that this condition is trivially satisfied in the two cases of Example \ref{ExTwoCases0}.
\begin{condition}\label{CondPTR}
Let
\begin{eqnarray*}
\Omega C\overset{e}{\longrightarrow}A\overset{f}{\longrightarrow}B\overset{g}{\longrightarrow}C\overset{h}{\longrightarrow}\Sigma A,\\
\Omega C^{\prime}\overset{e^{\prime}}{\longrightarrow}A^{\prime}\overset{f^{\prime}}{\longrightarrow}B^{\prime}\overset{g^{\prime}}{\longrightarrow}C^{\prime}\overset{h^{\prime}}{\longrightarrow}\Sigma A^{\prime}
\end{eqnarray*}
be extensions.
\begin{enumerate}
\item[(AC1)] If $c\in\mathcal{C}(C,C^{\prime})$ satisfies $h^{\prime}\circ c=0$ and $c\circ g=0$, then there exists $c^{\prime}\in\mathcal{C}(C,B^{\prime})$ such that $g^{\prime}\circ c^{\prime}=c$.
\item[(AC2)] If $a\in\mathcal{C}(A,A^{\prime})$ satisfies $f^{\prime}\circ a=0$ and $a\circ e=0$, then there exists $a^{\prime}\in\mathcal{C}(B,A^{\prime})$ such that $a^{\prime}\circ f=a$.
\end{enumerate}
\end{condition}

\begin{rem}If we impose the following conditions {\rm (1)} and {\rm (2)} on $\mathcal{C}$ (cf. \cite{BR}), then Condition \ref{CondPTR} is satisfied.
\begin{enumerate}
\item There exists an adjoint natural isomorphism
\[ \varphi_{A,B}\colon\mathcal{C}(\Sigma A,B)\overset{\cong}{\longrightarrow}\mathcal{C}(A,\Omega B) \quad (A,B\in\mathcal{C}). \]
\item Let
$A\overset{f}{\longrightarrow}B\overset{g}{\longrightarrow}C\overset{h}{\longrightarrow}\Sigma A$ and
$\Omega C^{\prime}\overset{e^{\prime}}{\longrightarrow}A^{\prime}\overset{f^{\prime}}{\longrightarrow}B^{\prime}\overset{g^{\prime}}{\longrightarrow}C^{\prime}$
be any object in $\triangleright$ and $\triangleleft$, respectively.

For any $a\in\mathcal{C}(A,\Omega C^{\prime})$ and $b\in\mathcal{C}(B,A^{\prime})$ satisfying $b\circ f=e^{\prime}\circ a$, there exists $c\in\mathcal{C}(C,B^{\prime})$ such that $c\circ g=f^{\prime}\circ b$ and $\varphi_{A,C^{\prime}}^{-1}(a)\circ\circ h=g^{\prime}\circ c$.

For any $c\in\mathcal{C}(C,B^{\prime})$ and $d\in\mathcal{C}(\Sigma A,C^{\prime})$ satisfying $d\circ h=g^{\prime}\circ c$, there exists $b\in\mathcal{C}(B,A^{\prime})$ such that $c\circ g=f^{\prime}\circ b$ and $b\circ f=e^{\prime}\circ\varphi_{A,C^{\prime}}(d)$.
\[
\xy
(-18,6)*+{A}="0";
(-6,6)*+{B}="2";
(6,6)*+{C}="4";
(18,6)*+{\Sigma A}="6";
(-18,-6)*+{\Omega C^{\prime}}="10";
(-6,-6)*+{A^{\prime}}="12";
(6,-6)*+{B^{\prime}}="14";
(18,-6)*+{C^{\prime}}="16";
{\ar^{f} "0";"2"};
{\ar^{g} "2";"4"};
{\ar^{h} "4";"6"};
{\ar_{e^{\prime}} "10";"12"};
{\ar_{f^{\prime}} "12";"14"};
{\ar_{g^{\prime}} "14";"16"};
{\ar_{} "0";"10"};
{\ar^{b} "2";"12"};
{\ar^{c} "4";"14"};
{\ar^{} "6";"16"};
{\ar@{}|\circlearrowright "0";"12"};
{\ar@{}|\circlearrowright "2";"14"};
{\ar@{}|\circlearrowright "4";"16"};
\endxy
\]
\end{enumerate}
\end{rem}

In the rest, $\mathcal{C}$ is assumed to satisfy Condition \ref{CorCor}. 
First, we construct the shift functor.
\begin{lem}\label{LemForSuspension}
Let 
\[
\xy
(-24,6)*+{\Omega Z}="-2";
(-12,6)*+{X}="0";
(0,6)*+{Y}="2";
(12,6)*+{Z}="4";
(24,6)*+{\Sigma X}="6";
(-24,-6)*+{\Omega S}="-12";
(-12,-6)*+{M}="10";
(0,-6)*+{I}="12";
(12,-6)*+{S}="14";
(24,-6)*+{\Sigma M}="16";
{\ar^{e} "-2";"0"};
{\ar^{f} "0";"2"};
{\ar^{g} "2";"4"};
{\ar^{h} "4";"6"};
{\ar_{\delta} "-12";"10"};
{\ar_{\alpha} "10";"12"};
{\ar_{\beta} "12";"14"};
{\ar_{\gamma} "14";"16"};
{\ar_{} "-2";"-12"};
{\ar^{x} "0";"10"};
{\ar^{y} "2";"12"};
{\ar^{z} "4";"14"};
{\ar^{\Sigma x} "6";"16"};
{\ar@{}|\circlearrowright "-2";"10"};
{\ar@{}|\circlearrowright "0";"12"};
{\ar@{}|\circlearrowright "2";"14"};
{\ar@{}|\circlearrowright "4";"16"};
\endxy
\ ,\ 
\xy
(-24,6)*+{\Omega Z}="-2";
(-12,6)*+{X}="0";
(0,6)*+{Y}="2";
(12,6)*+{Z}="4";
(24,6)*+{\Sigma X}="6";
(-24,-6)*+{\Omega S}="-12";
(-12,-6)*+{M}="10";
(0,-6)*+{I}="12";
(12,-6)*+{S}="14";
(24,-6)*+{\Sigma M}="16";
{\ar^{e} "-2";"0"};
{\ar^{f} "0";"2"};
{\ar^{g} "2";"4"};
{\ar^{h} "4";"6"};
{\ar_{\delta} "-12";"10"};
{\ar_{\alpha} "10";"12"};
{\ar_{\beta} "12";"14"};
{\ar_{\gamma} "14";"16"};
{\ar_{} "-2";"-12"};
{\ar^{x^{\prime}} "0";"10"};
{\ar^{y^{\prime}} "2";"12"};
{\ar^{z^{\prime}} "4";"14"};
{\ar^{\Sigma x^{\prime}} "6";"16"};
{\ar@{}|\circlearrowright "-2";"10"};
{\ar@{}|\circlearrowright "0";"12"};
{\ar@{}|\circlearrowright "2";"14"};
{\ar@{}|\circlearrowright "4";"16"};
\endxy
\]
be morphisms of conflations, with $I\in\mathcal{I}_{\mathcal{D}}$. Then $\underline{x}=\underline{x^{\prime}}$ in $\mathcal{Z}/\mathcal{I}_{\mathcal{D}}$ implies $\underline{z}=\underline{z^{\prime}}$ in $\mathcal{Z}/\mathcal{I}_{\mathcal{D}}$.
\end{lem}
\begin{proof}
Obviously, it suffices to show that $\underline{x}=0$ implies $\underline{z}=0$ in the first diagram.

Since $\underline{x}=0$, there exist $I_0\in\mathcal{I}_{\mathcal{D}}$, $x_1\in\mathcal{Z}(X,I_0)$ and $x_2\in\mathcal{Z}(I_0,M)$ such that $x=x_2\circ x_1$.
Since $I_0\in\mathcal{I}_{\mathcal{D}}$ and $f$ is an inflation, there exists $x_3\in\mathcal{Z}(Y,I_0)$ such that $x_3\circ f=x_1$.
Thus we have $x\circ e=x_2\circ x_3\circ f\circ e=0$, which implies
\[ (\Sigma x)\circ h=-(\Sigma x)\circ\psi(e)=-\psi(x\circ e)=0. \]
Put $\eta=y-\alpha\circ x_2\circ x_3$.
By $\eta\circ f=0$, there exists $s\in\mathcal{Z}(Z,I)$ such that $s\circ g=\eta$. Thus we have
\begin{eqnarray*}
\gamma\circ(z-\beta\circ s)=\gamma\circ z=(\Sigma x)\circ h=0,\\
(z-\beta\circ s)\circ g=z\circ g-\beta\circ y=0.
\end{eqnarray*}
By {\rm (AC1)}, there exists $t\in\mathcal{Z}(Z,I)$ such that $z-\beta\circ s=\beta\circ t$, namely $z=\beta\circ(s+t)$.
\end{proof}

\begin{const}
Assume $(\mathcal{C},\mathcal{Z},\mathcal{D})$ has enough injectives. For any $X\in\mathcal{Z}$, take a conflation
\[ \Omega S_X\overset{\delta_X}{\longrightarrow}X\overset{\alpha_X}{\longrightarrow}I_X\overset{\beta_X}{\longrightarrow}S_X\overset{\gamma_X}{\longrightarrow}\Sigma X \]
with $I_X\in\mathcal{I}_{\mathcal{D}}$.
Define $S(X)=SX$ to be the image of $S_X$ in $\mathcal{Z}/\mathcal{I}_{\mathcal{D}}$.

For any morphism $f\in\mathcal{Z}(X,Y)$, take a conflation
\[ \Omega S_Y\overset{\delta_Y}{\longrightarrow}Y\overset{\alpha_Y}{\longrightarrow}I_Y\overset{\beta_Y}{\longrightarrow}S_Y\overset{\gamma_Y}{\longrightarrow}\Sigma Y \]
similarly for $Y$. Since $\alpha_X$ is an inflation and $I_Y\in\mathcal{I}_{\mathcal{D}}$, there exists $I_f\in\mathcal{Z}(I_X,I_Y)$ such that $I_f\circ\alpha_X=\alpha_Y\circ f$. By {\rm (RTR3)}, there exists $S_f\in\mathcal{Z}(S_X,S_Y)$ such that $(f,I_f,S_f)$ is a morphism of conflations.
\[
\xy
(-28,6)*+{\Omega S_X}="-2";
(-14,6)*+{X}="0";
(0,6)*+{I_X}="2";
(14,6)*+{S_X}="4";
(28,6)*+{\Sigma X}="6";
(-28,-6)*+{\Omega S_Y}="-12";
(-14,-6)*+{Y}="10";
(0,-6)*+{I_Y}="12";
(14,-6)*+{S_Y}="14";
(28,-6)*+{\Sigma Y}="16";
{\ar^{\delta_X} "-2";"0"};
{\ar^{\alpha_X} "0";"2"};
{\ar^{\beta_X} "2";"4"};
{\ar^{\gamma_X} "4";"6"};
{\ar_{\delta_Y} "-12";"10"};
{\ar_{\alpha_Y} "10";"12"};
{\ar_{\beta_Y} "12";"14"};
{\ar_{\gamma_Y} "14";"16"};
{\ar_{\Omega S_f} "-2";"-12"};
{\ar^{f} "0";"10"};
{\ar^{I_f} "2";"12"};
{\ar^{S_f} "4";"14"};
{\ar^{\Sigma f} "6";"16"};
{\ar@{}|\circlearrowright "-2";"10"};
{\ar@{}|\circlearrowright "0";"12"};
{\ar@{}|\circlearrowright "2";"14"};
{\ar@{}|\circlearrowright "4";"16"};
\endxy
\]
For any $\underline{f}\in\mathcal{Z}/\mathcal{I}_{\mathcal{D}}(X,Y)$, define $S\underline{f}$ to be the image $\underline{S_f}$ of $S_f$ in $\mathcal{Z}/\mathcal{I}_{\mathcal{D}}$. This is well-defined by Lemma \ref{LemForSuspension}, and the following proposition holds.
\end{const}

\begin{prop}
$S\colon\mathcal{Z}/\mathcal{I}_{\mathcal{D}}\rightarrow\mathcal{Z}/\mathcal{I}_{\mathcal{D}}$ gives an additive functor.
\end{prop}
\begin{proof}
This immediately follows from Lemma \ref{LemForSuspension}.
\end{proof}

\begin{rem}
Dually, if $(\mathcal{C},\mathcal{Z},\mathcal{D})$ has enough projectives, then we have an additive functor $S^{\ast}\colon\mathcal{Z}/\mathcal{P}_{\mathcal{D}}\rightarrow\mathcal{Z}/\mathcal{P}_{\mathcal{D}}$, defined by a conflation
\[ \Omega X\rightarrow S^{\ast}X\rightarrow P_X\rightarrow X\rightarrow\Sigma S^{\ast}X \]
for any $X\in\mathcal{Z}$, where $P_X\in\mathcal{P}_{\mathcal{D}}$.
\end{rem}

\begin{prop}
If $(\mathcal{C},\mathcal{Z},\mathcal{D})$ is Frobenius, then $S$ and $S^{\ast}$ are quasi-inverses.
\end{prop}
\begin{proof}
This follows immediately from the definitions of $S$ and $S^{\ast}$.
\end{proof}

In the rest, $(\mathcal{C},\mathcal{Z},\mathcal{D})$ is assumed to be Frobenius.
Next, we define the class of distinguished triangles on $\mathcal{Z}/\mathcal{I}_D$.
\begin{dfn}
Let $\Omega Z\overset{e}{\longrightarrow}X\overset{f}{\longrightarrow}Y\overset{g}{\longrightarrow}Z\overset{h}{\longrightarrow}\Sigma X$ be any conflation, and take a conflation $\Omega S_X\overset{\delta_X}{\longrightarrow}X\overset{\alpha_X}{\longrightarrow}F_X\overset{\beta_X}{\longrightarrow}S_X\overset{\gamma_X}{\longrightarrow}\Sigma X$ where $I_X\in\mathcal{I}_{\mathcal{D}}$.

If there exist $p\in\mathcal{Z}(Y,I_X)$ and $q\in\mathcal{Z}(Z,S_X)$ satisfying
\begin{eqnarray*}
p\circ f=\alpha_X,\ \ q\circ g=\beta_X\circ p,\ \ \gamma_X\circ q=h
\end{eqnarray*}
(namely, $(\mathrm{id},p,q)$ is a morphism of conflations)
\[
\xy
(-24,6)*+{X}="0";
(-8,6)*+{Y}="2";
(8,6)*+{Z}="4";
(24,6)*+{\Sigma X}="6";
(-24,-6)*+{X}="10";
(-8,-6)*+{I_X}="12";
(8,-6)*+{S_X}="14";
(24,-6)*+{\Sigma X}="16";
{\ar^{f} "0";"2"};
{\ar^{g} "2";"4"};
{\ar^{h} "4";"6"};
{\ar_{\alpha_X} "10";"12"};
{\ar_{\beta_X} "12";"14"};
{\ar_{\gamma_X} "14";"16"};
{\ar@{=} "0";"10"};
{\ar^{p} "2";"12"};
{\ar^{q} "4";"14"};
{\ar@{=} "6";"16"};
{\ar@{}|\circlearrowright "0";"12"};
{\ar@{}|\circlearrowright "2";"14"};
{\ar@{}|\circlearrowright "4";"16"};
\endxy,
\]
then we call the sequence
\[ X\overset{\underline{f}}{\longrightarrow}Y\overset{\underline{g}}{\longrightarrow}Z\overset{\underline{q}}{\longrightarrow}SX \]
a {\it standard triangle}.
Remark that by {\rm (RTR3)} and the injectivity of $I_X$, there exists at least one such pair of morphisms $(p,q)$. We define the class of distinguished triangles $\triangle$ to be the category of triangles
\begin{equation}
X\rightarrow Y\rightarrow Z\rightarrow SZ
\label{triangle}
\end{equation}
in $\mathcal{Z}/\mathcal{I}_{\mathcal{D}}$, which are isomorphic to standard triangles.
\end{dfn}

In the rest, we show that $(\mathcal{Z}/\mathcal{I}_{\mathcal{D}},S,\triangle)$ is a triangulated category.

\begin{prop}\label{PropTR1}
$(\mathcal{Z}/\mathcal{I}_{\mathcal{D}},S,\triangle)$ satisfies {\rm (TR1)}.
\end{prop}
\begin{proof}$\ \ $
\begin{enumerate}
\item By definition, every diagram (\ref{triangle}) isomorphic to an object in $\triangle$ also belongs to $\triangle$.

\item Let $f\in\mathcal{Z}(X,Y)$ be any morphism. Take a conflation
\[ \Omega S_X\overset{\delta_X}{\longrightarrow}X\overset{\alpha_X}{\longrightarrow}I_X\overset{\beta_X}{\longrightarrow}S_X\overset{\gamma_X}{\longrightarrow}\Sigma X \]
with $I_X\in\mathcal{I}_{\mathcal{D}}$, and put $f_X=(f,-\alpha_X)$. By Corollary \ref{CorExtension}, Lemma \ref{LemSPL} and Proposition \ref{PropOCT}, $f_X\colon X\rightarrow Y\oplus I_X$ becomes an inflation.
In fact, by Corollary \ref{CorExtension} and Lemma \ref{LemSPL}, there exists an extension
\[ \Omega C_f\rightarrow X\overset{f_X}{\longrightarrow}Y\oplus I_X\overset{c_f}{\longrightarrow}C_f\overset{\ell_f}{\longrightarrow}\Sigma_X, \]
and applying Proposition \ref{PropOCT} to the following diagram $(\ref{***Diag})$ of extensions, we obtain an extension
\[ \Omega S_X\rightarrow Y\rightarrow C_f\overset{{}^{\exists}q}{\longrightarrow} S_X\rightarrow \Sigma Y, \]and thus $C_f\in\mathcal{Z}$ by the extension-closedness of $\mathcal{Z}$.
\begin{equation}
\xy
(-32,0)*+{\Omega C_f}="0";
(-16,0)*+{X}="2";
(0,0)*+{Y\oplus I_X}="4";
(16,0)*+{C_f}="6";
(32,0)*+{\Sigma X}="8";
(0,24)*+{\Sigma Y}="10";
(0,12)*+{I_X}="12";
(10,20)*+{}="13";
(16,12)*+{S_X}="14";
(-16,-12)*+{\Omega S_X}="16";
(0,-12)*+{Y}="18";
(0,-24)*+{\Omega I_X}="20";
(-10,8)*+{}="22";
(10,-8)*+{}="24";
(26,8)*+{}="26";
{\ar^{} "0";"2"};
{\ar_{f_X} "2";"4"};
{\ar_{c_f} "4";"6"};
{\ar_{\ell_f} "6";"8"};
{\ar^{} "16";"2"};
{\ar^{\alpha_X} "2";"12"};
{\ar_{\beta_X} "12";"14"};
{\ar^{\gamma_X} "14";"8"};
{\ar^{} "20";"18"};
{\ar^{i_Y} "18";"4"};
{\ar|{_{-p_{I_X}}} "4";"12"};
{\ar^{} "12";"10"};
{\ar@{-->}^{{\exists}q} "6";"14"};
{\ar@{-->}_{} "16";"18"};
{\ar@{-->}_{} "18";"6"};
{\ar@{-->}_{} "14";"10"};
{\ar@{}|\circlearrowright "6";"12"};
{\ar@{}|\circlearrowright "4";"22"};
{\ar@{}|\circlearrowright "4";"24"};
{\ar@{}|\circlearrowright "6";"26"};
{\ar@{}|\circlearrowright "12";"13"};
\endxy
\label{***Diag}
\end{equation}

Let $C(f)$ denote the image of $C_f$ in $\mathcal{Z}/\mathcal{I}_{\mathcal{D}}$.
Then the above diagram means
\[ X\overset{\underline{f_X}}{\longrightarrow}Y\oplus I_X\overset{\underline{c_f}}{\longrightarrow}C(f)\overset{\underline{q}}{\longrightarrow}SX \]
is a standard triangle.
If we put $g=c_f\circ i_Y$ where $i_Y\colon Y\hookrightarrow Y\oplus I_X$ is the inclusion, then $X\overset{\underline{f}}{\longrightarrow}Y\overset{\underline{g}}{\longrightarrow}C(f)\overset{\underline{q}}{\longrightarrow}SX$ becomes isomorphic to this standard triangle. 

\item By {\rm (RTR1)}, {\rm (RTR2)} and {\rm (LTR1)},
\[ 0=\Omega 0\rightarrow X\overset{\mathrm{id}}{\longrightarrow}X\rightarrow 0\rightarrow\Sigma X \]
is a conflation, and it immediately follows that the triangle
\[ X\overset{\mathrm{id}}{\longrightarrow}X\rightarrow 0\rightarrow SX \]
belongs to $\triangle$.

\end{enumerate}
\end{proof}

\begin{prop}\label{PropTR2}
$(\mathcal{Z}/\mathcal{I}_{\mathcal{D}},S,\triangle)$ satisfies {\rm (TR2)}.
\end{prop}
\begin{proof}
It suffices to show, for any distinguished triangle
\[ X\overset{\underline{f}}{\longrightarrow}Y\overset{\underline{g}}{\longrightarrow}Z\overset{\underline{q}}{\longrightarrow}SX \]
arising from a morphism of conflations
\[
\xy
(-24,6)*+{\Omega Z}="-2";
(-12,6)*+{X}="0";
(0,6)*+{Y}="2";
(12,6)*+{Z}="4";
(24,6)*+{\Sigma X}="6";
(-24,-6)*+{\Omega S_X}="-12";
(-12,-6)*+{X}="10";
(0,-6)*+{I_X}="12";
(12,-6)*+{S_X}="14";
(24,-6)*+{\Sigma X}="16";
{\ar^{e} "-2";"0"};
{\ar^{f} "0";"2"};
{\ar^{g} "2";"4"};
{\ar^{h} "4";"6"};
{\ar_{\delta_X} "-12";"10"};
{\ar_{\alpha_X} "10";"12"};
{\ar_{\beta_X} "12";"14"};
{\ar_{\gamma_X} "14";"16"};
{\ar_{\Omega q} "-2";"-12"};
{\ar@{=} "0";"10"};
{\ar^{p} "2";"12"};
{\ar^{q} "4";"14"};
{\ar@{=} "6";"16"};
{\ar@{}|\circlearrowright "-2";"10"};
{\ar@{}|\circlearrowright "0";"12"};
{\ar@{}|\circlearrowright "2";"14"};
{\ar@{}|\circlearrowright "4";"16"};
\endxy,
\]
the shifted triangle
\[ Y\overset{\underline{g}}{\longrightarrow}Z\overset{{}\underline{q}}{\longrightarrow}SX\overset{-S\underline{f}}{\longrightarrow}SY \]
also becomes a distinguished triangle.

We may replace $\Omega Z\overset{e}{\longrightarrow}X\overset{f}{\longrightarrow}Y\overset{g}{\longrightarrow}Z\overset{h}{\longrightarrow}\Sigma X$ by the conflation $\Omega C_f\rightarrow X\overset{f_X}{\longrightarrow}Y\oplus I_X\overset{c_f}{\longrightarrow}C_f\overset{\ell_f}{\longrightarrow}\Sigma X$ constructed in the proof of Proposition \ref{PropTR1}. Recall that $f_X=(f,-\alpha_X)=i_Y\circ f-i_{I_X}\circ\alpha_X$ where $i_Y$ and $i_{I_X}$ are the inclusions into $Y\oplus I_X$.

Take conflations
\begin{eqnarray*}
\Omega S_X\overset{\delta_X}{\longrightarrow}X\overset{\alpha_X}{\longrightarrow}I_X\overset{\beta_X}{\longrightarrow}S_X\overset{\gamma_X}{\longrightarrow}\Sigma X,\\
\Omega I_X\overset{0}{\longrightarrow}Y\overset{i_Y}{\longrightarrow}Y\oplus I_X\overset{-p_{I_X}}{\longrightarrow}I_X\overset{0}{\longrightarrow}\Sigma Y.
\end{eqnarray*}
By Proposition \ref{PropOCT}, there exists $k\in\mathcal{C}(\Omega S_X,Y)$ and $\nu\in\mathcal{Z}(C_f,S_X)$ such that
\[ \Omega S_X\overset{k}{\longrightarrow}Y\overset{\mu}{\longrightarrow}C_f\overset{\nu}{\longrightarrow}S_X\rightarrow\Sigma Y \]
is a conflation, where $\mu=c_f\circ i_Y$, and
\begin{eqnarray*}
\nu\circ c_f=-\beta_X\circ p_{I_X}&,&\gamma_X\circ\nu=\ell_f,\\
-\psi_{S_X,Y}(k)\circ\beta_X=0&,&f_X\circ\delta_X+i_Y\circ k=0.\\
\end{eqnarray*}
\[
\xy
(-32,0)*+{\Omega C_f}="0";
(-16,0)*+{X}="2";
(0,0)*+{Y\oplus I_X}="4";
(16,0)*+{C_f}="6";
(32,0)*+{\Sigma X}="8";
(0,24)*+{\Sigma Y}="10";
(0,12)*+{I_X}="12";
(10,20)*+{}="13";
(16,12)*+{S_X}="14";
(-16,-12)*+{\Omega S_X}="16";
(0,-12)*+{Y}="18";
(0,-24)*+{\Omega I_X}="20";
(-10,8)*+{}="22";
(10,-8)*+{}="24";
(26,8)*+{}="26";
{\ar^{} "0";"2"};
{\ar_{f_X} "2";"4"};
{\ar_{c_f} "4";"6"};
{\ar_{\ell_f} "6";"8"};
{\ar^{\delta_X} "16";"2"};
{\ar^{\alpha_X} "2";"12"};
{\ar_{\beta_X} "12";"14"};
{\ar^{\gamma_X} "14";"8"};
{\ar^{0} "20";"18"};
{\ar^{i_Y} "18";"4"};
{\ar|{_{-p_{I_X}}} "4";"12"};
{\ar^{0} "12";"10"};
{\ar@{-->}^{{}^\exists \nu} "6";"14"};
{\ar@{-->}_{{}^{\exists}k} "16";"18"};
{\ar@{-->}_{\mu} "18";"6"};
{\ar@{-->}_{-\psi_{S_X,Y}(k)} "14";"10"};
{\ar@{}|\circlearrowright "6";"12"};
{\ar@{}|\circlearrowright "4";"22"};
{\ar@{}|\circlearrowright "4";"24"};
{\ar@{}|\circlearrowright "6";"26"};
{\ar@{}|\circlearrowright "12";"13"};
\endxy
\]
\begin{claim}\label{ClaimForTR2}
We have a morphism of conflations
\[
\xy
(-28,7)*+{\Omega S_X}="-2";
(-14,7)*+{X}="0";
(0,7)*+{I_X}="2";
(14,7)*+{S_X}="4";
(28,7)*+{\Sigma X}="6";
(-28,-7)*+{\Omega S_X}="-12";
(-14,-7)*+{Y}="10";
(0,-7)*+{C_f}="12";
(14,-7)*+{S_X}="14";
(28,-7)*+{\Sigma Y}="16";
{\ar^{\delta_X} "-2";"0"};
{\ar^{\alpha_X} "0";"2"};
{\ar^{\beta_X} "2";"4"};
{\ar^{\gamma_X} "4";"6"};
{\ar_{k} "-12";"10"};
{\ar_{\mu} "10";"12"};
{\ar_{\nu} "12";"14"};
{\ar_{} "14";"16"};
{\ar_{-\mathrm{id}} "-2";"-12"};
{\ar^{f} "0";"10"};
{\ar|*+{_{c_f\circ i_{I_X}}} "2";"12"};
{\ar|*+{_{-\mathrm{id}}} "4";"14"};
{\ar^{\Sigma f} "6";"16"};
{\ar@{}|\circlearrowright "-2";"10"};
{\ar@{}|\circlearrowright "0";"12"};
{\ar@{}|\circlearrowright "2";"14"};
{\ar@{}|\circlearrowright "4";"16"};
\endxy
\]
\end{claim}
\begin{proof}[Proof of Claim \ref{ClaimForTR2}]
This immediately follows from
\[ f\circ\delta_X=p_Y\circ f_X\circ\delta_X=-p_Y\circ i_Y\circ k=-k, \]
\begin{eqnarray*}
c_f\circ i_{I_X}\circ\alpha_X&=&c_f\circ i_{I_X}\circ(-p_{I_X})\circ f_X\\
&=&c_f\circ(i_Y\circ p_Y\circ f_X-f_X)\\
&=&c_f\circ i_Y\circ f,
\end{eqnarray*}
\[ \nu\circ c_f\circ i_{I_X}=-\beta_X\circ p_{I_X}\circ i_{I_X}=-\beta_X. \]
\end{proof}
If we take a conflation $\Omega S_Y\overset{\delta_Y}{\longrightarrow}Y\overset{\alpha_Y}{\longrightarrow}I_Y\overset{\beta_Y}{\longrightarrow}S_Y\overset{\gamma_Y}{\longrightarrow}\Sigma Y$ where $I_Y\in\mathcal{I}_{\mathcal{D}}$, then there exist $u\in\mathcal{Z}(C_f,I_Y)$ and $v\in\mathcal{Z}(S_X,S_Y)$ such that $(\mathrm{id}_Y,p,q)$ is a morphism of conflations.
\begin{equation}
\xy
(-28,7)*+{\Omega S_X}="-2";
(-14,7)*+{Y}="0";
(0,7)*+{C_f}="2";
(14,7)*+{S_X}="4";
(28,7)*+{\Sigma Y}="6";
(-28,-7)*+{\Omega S_Y}="-12";
(-14,-7)*+{Y}="10";
(0,-7)*+{I_Y}="12";
(14,-7)*+{S_Y}="14";
(28,-7)*+{\Sigma Y}="16";
{\ar^{k} "-2";"0"};
{\ar^{\mu} "0";"2"};
{\ar^{\nu} "2";"4"};
{\ar^{} "4";"6"};
{\ar_{\delta_Y} "-12";"10"};
{\ar_{\alpha_Y} "10";"12"};
{\ar_{\beta_Y} "12";"14"};
{\ar_{\gamma_Y} "14";"16"};
{\ar_{\Omega q} "-2";"-12"};
{\ar@{=} "0";"10"};
{\ar^{u} "2";"12"};
{\ar^{v} "4";"14"};
{\ar@{=} "6";"16"};
{\ar@{}|\circlearrowright "-2";"10"};
{\ar@{}|\circlearrowright "0";"12"};
{\ar@{}|\circlearrowright "2";"14"};
{\ar@{}|\circlearrowright "4";"16"};
\endxy
\label{CompDiag}
\end{equation}
By definition, we have a standard triangle in $\triangle$
\[ Y\overset{\underline{\mu}}{\longrightarrow}C(f)\overset{\underline{\nu}}{\longrightarrow}SX\overset{\underline{v}}{\longrightarrow}SY. \]

Composing $(\ref{CompDiag})$ with the morphism obtained in Claim \ref{ClaimForTR2}, we obtain the following morphism of conflations, which means $S\underline{f}=-\underline{v}$.
\[
\xy
(-28,7)*+{\Omega S_X}="-2";
(-14,7)*+{X}="0";
(0,7)*+{I_X}="2";
(14,7)*+{S_X}="4";
(28,7)*+{\Sigma X}="6";
(-28,-7)*+{\Omega S_Y}="-12";
(-14,-7)*+{Y}="10";
(0,-7)*+{I_Y}="12";
(14,-7)*+{S_Y}="14";
(28,-7)*+{\Sigma Y}="16";
{\ar^{\delta_X} "-2";"0"};
{\ar^{\alpha_X} "0";"2"};
{\ar^{\beta_X} "2";"4"};
{\ar^{\gamma_X} "4";"6"};
{\ar_{\delta_Y} "-12";"10"};
{\ar_{\alpha_Y} "10";"12"};
{\ar_{\beta_Y} "12";"14"};
{\ar_{\gamma_Y} "14";"16"};
{\ar_{-\Omega q} "-2";"-12"};
{\ar^{f} "0";"10"};
{\ar|{_{u\circ c_f\circ i_{I_X}}} "2";"12"};
{\ar^{-v} "4";"14"};
{\ar^{\Sigma f} "6";"16"};
{\ar@{}|\circlearrowright "-2";"10"};
{\ar@{}|\circlearrowright "0";"12"};
{\ar@{}|\circlearrowright "2";"14"};
{\ar@{}|\circlearrowright "4";"16"};
\endxy
\]

\end{proof}
\begin{lem}\label{LemShift}
Let
\[ X\overset{\underline{f}}{\longrightarrow}Y\overset{\underline{g}}{\longrightarrow}Z\overset{\underline{q}}{\longrightarrow}SX \]and 
\[ X^{\prime}\overset{\underline{f^{\prime}}}{\longrightarrow}Y^{\prime}\overset{\underline{g^{\prime}}}{\longrightarrow}Z^{\prime}\overset{\underline{q^{\prime}}}{\longrightarrow}SX^{\prime} \]
be standard triangles in $\mathcal{Z}/\mathcal{I}_{\mathcal{D}}$ obtained from 
\[
\xy
(-24,6)*+{\Omega Z}="-2";
(-12,6)*+{X}="0";
(0,6)*+{Y}="2";
(12,6)*+{Z}="4";
(24,6)*+{\Sigma X}="6";
(-24,-6)*+{\Omega S_X}="-12";
(-12,-6)*+{X}="10";
(0,-6)*+{I_X}="12";
(12,-6)*+{S_X}="14";
(24,-6)*+{\Sigma X}="16";
{\ar^{e} "-2";"0"};
{\ar^{f} "0";"2"};
{\ar^{g} "2";"4"};
{\ar^{h} "4";"6"};
{\ar^{} "-12";"10"};
{\ar_{\alpha_X} "10";"12"};
{\ar_{\beta_X} "12";"14"};
{\ar_{\gamma_X} "14";"16"};
{\ar^{} "-2";"-12"};
{\ar@{=} "0";"10"};
{\ar^{p} "2";"12"};
{\ar^{q} "4";"14"};
{\ar@{=} "6";"16"};
{\ar@{}|\circlearrowright "-2";"10"};
{\ar@{}|\circlearrowright "0";"12"};
{\ar@{}|\circlearrowright "2";"14"};
{\ar@{}|\circlearrowright "4";"16"};
\endxy
\ \text{and}\ 
\xy
(-24,6)*+{\Omega Z^{\prime}}="-2";
(-12,6)*+{X^{\prime}}="0";
(0,6)*+{Y^{\prime}}="2";
(12,6)*+{Z^{\prime}}="4";
(24,6)*+{\Sigma X^{\prime}}="6";
(-24,-6)*+{\Omega S_{X^{\prime}}}="-12";
(-12,-6)*+{X^{\prime}}="10";
(0,-6)*+{I_{X^{\prime}}}="12";
(12,-6)*+{S_{X^{\prime}}}="14";
(24,-6)*+{\Sigma X^{\prime}}="16";
{\ar^{e^{\prime}} "-2";"0"};
{\ar^{f^{\prime}} "0";"2"};
{\ar^{g^{\prime}} "2";"4"};
{\ar^{h^{\prime}} "4";"6"};
{\ar^{} "-12";"10"};
{\ar_{\alpha_{X^{\prime}}} "10";"12"};
{\ar_{\beta_{X^{\prime}}} "12";"14"};
{\ar_{\gamma_{X^{\prime}}} "14";"16"};
{\ar^{} "-2";"-12"};
{\ar@{=} "0";"10"};
{\ar^{p^{\prime}} "2";"12"};
{\ar^{q^{\prime}} "4";"14"};
{\ar@{=} "6";"16"};
{\ar@{}|\circlearrowright "-2";"10"};
{\ar@{}|\circlearrowright "0";"12"};
{\ar@{}|\circlearrowright "2";"14"};
{\ar@{}|\circlearrowright "4";"16"};
\endxy
.
\]
If we are given a morphism of conflations
\[
\xy
(-40,6)*+{\Omega Z}="-2";
(-24,6)*+{X}="0";
(-8,6)*+{Y}="2";
(8,6)*+{Z}="4";
(24,6)*+{\Sigma X}="6";
(-40,-6)*+{\Omega Z^{\prime}}="-12";
(-24,-6)*+{X^{\prime}}="10";
(-8,-6)*+{Y^{\prime}}="12";
(8,-6)*+{Z^{\prime}}="14";
(24,-6)*+{\Sigma X^{\prime}}="16";
{\ar^{e} "-2";"0"};
{\ar^{f} "0";"2"};
{\ar^{g} "2";"4"};
{\ar^{h} "4";"6"};
{\ar_{e^{\prime}} "-12";"10"};
{\ar_{f^{\prime}} "10";"12"};
{\ar_{g^{\prime}} "12";"14"};
{\ar_{h^{\prime}} "14";"16"};
{\ar_{\Omega z} "-2";"-12"};
{\ar^{x} "0";"10"};
{\ar^{y} "2";"12"};
{\ar^{z} "4";"14"};
{\ar^{\Sigma x} "6";"16"};
{\ar@{}|\circlearrowright "-2";"10"};
{\ar@{}|\circlearrowright "0";"12"};
{\ar@{}|\circlearrowright "2";"14"};
{\ar@{}|\circlearrowright "4";"16"};
\endxy
,
\]
then we obtain the following morphism in $\triangle$.
\[
\xy
(-24,6)*+{X}="0";
(-8,6)*+{Y}="2";
(8,6)*+{Z}="4";
(24,6)*+{SX}="6";
(-24,-6)*+{X^{\prime}}="10";
(-8,-6)*+{Y^{\prime}}="12";
(8,-6)*+{Z^{\prime}}="14";
(24,-6)*+{SX^{\prime}}="16";
{\ar^{\underline{f}} "0";"2"};
{\ar^{\underline{g}} "2";"4"};
{\ar^{\underline{q}} "4";"6"};
{\ar_{\underline{f^{\prime}}} "10";"12"};
{\ar_{\underline{g^{\prime}}} "12";"14"};
{\ar_{\underline{q^{\prime}}} "14";"16"};
{\ar_{\underline{x}} "0";"10"};
{\ar^{\underline{y}} "2";"12"};
{\ar^{\underline{z}} "4";"14"};
{\ar^{S\underline{x}} "6";"16"};
{\ar@{}|\circlearrowright "0";"12"};
{\ar@{}|\circlearrowright "2";"14"};
{\ar@{}|\circlearrowright "4";"16"};
\endxy
\]
\end{lem}
\begin{proof}
It suffices to show $\underline{q^{\prime}}\circ\underline{z}=(S\underline{x})\circ\underline{q}$. By the definition of $S_x$, we have $(\Sigma x)\circ\gamma_X=\gamma_{X^{\prime}}\circ{S_x}$.
Since $(I_x\circ p-p^{\prime}\circ y)\circ f=I_x\circ p\circ f-p^{\prime}\circ y\circ f=I_x\circ\alpha_X-\alpha_{X^{\prime}}\circ x=0$
\[
\xy
(-24,6)*+{\Omega S_X}="-2";
(-12,6)*+{X}="0";
(0,6)*+{I_X}="2";
(12,6)*+{S_X}="4";
(24,6)*+{\Sigma X}="6";
(-24,-6)*+{\Omega S_{X^{\prime}}}="-12";
(-12,-6)*+{X^{\prime}}="10";
(0,-6)*+{I_{X^{\prime}}}="12";
(12,-6)*+{S_{X^{\prime}}}="14";
(24,-6)*+{\Sigma M}="16";
{\ar^{\alpha_X} "0";"2"};
{\ar^{\beta_X} "2";"4"};
{\ar^{\gamma_X} "4";"6"};
{\ar_{\alpha_{X^{\prime}}} "10";"12"};
{\ar_{\beta_{X^{\prime}}} "12";"14"};
{\ar_{\gamma_{X^{\prime}}} "14";"16"};
{\ar^{} "-2";"0"};
{\ar^{x} "0";"10"};
{\ar^{I_x} "2";"12"};
{\ar^{S_x} "4";"14"};
{\ar^{\Sigma x} "6";"16"};
{\ar^{} "-12";"10"};
{\ar^{} "-2";"-12"};
{\ar@{}|\circlearrowright "-2";"10"};
{\ar@{}|\circlearrowright "0";"12"};
{\ar@{}|\circlearrowright "2";"14"};
{\ar@{}|\circlearrowright "4";"16"};
\endxy
,
\]
there exists $s\in\mathcal{Z}(Z,I_{X^{\prime}})$ such that $s\circ g=I_x\circ p-p^{\prime}\circ y$.
If we put $\zeta=S_x\circ q-q^{\prime}\circ z-\beta_{X^{\prime}}\circ s$, then $\zeta$ satisfies
\begin{eqnarray*}
\gamma_{X^{\prime}}\circ\zeta&=&\gamma_{X^{\prime}}\circ S_x\circ q-\gamma_{X^{\prime}}\circ q^{\prime}\circ z-\gamma_{X^{\prime}}\circ\beta_{X^{\prime}}\circ s\\
&=&(\Sigma x)\circ\gamma_X\circ q-h^{\prime}\circ z\\
&=&(\Sigma x)\circ h-(\Sigma x)\circ h\\
&=&0
\end{eqnarray*}
and
\begin{eqnarray*}
\zeta\circ g&=&S_x\circ q\circ g-q^{\prime}\circ z\circ g-\beta_{X^{\prime}}\circ s\circ g\\
&=&S_x\circ\beta_X\circ p-q^{\prime}\circ g^{\prime}\circ y-(\beta_{X^{\prime}}\circ I_x\circ p-\beta_{X^{\prime}}\circ p^{\prime}\circ y)\\
&=&\beta_{X^{\prime}}\circ I_x\circ p-\beta_{X^{\prime}}\circ p^{\prime}\circ y-(\beta_{X^{\prime}}\circ I_x\circ p-\beta_{X^{\prime}}\circ p^{\prime}\circ y)\\
&=&0.
\end{eqnarray*}
Thus by {\rm (AC1)}, there exists $t\in\mathcal{Z}(Z,I_{X^{\prime}})$ such that $\zeta=\beta_{X^{\prime}}\circ t$, i.e.,
\[ S_x\circ q-q^{\prime}\circ z=\beta_{X^{\prime}}\circ(s+t). \]
\[
\xy
(-30,6)*+{\Omega Z}="-2";
(-18,6)*+{X}="0";
(-6,6)*+{Y}="2";
(-2,2)*+{}="3";
(6,6)*+{Z}="4";
(18,6)*+{\Sigma X}="6";
(-30,-6)*+{\Omega S_{X^{\prime}}}="-12";
(-18,-6)*+{X^{\prime}}="10";
(-6,-6)*+{I_{X^{\prime}}}="12";
(6,-6)*+{S_{X^{\prime}}}="14";
(18,-6)*+{\Sigma X}="16";
{\ar^{} "-2";"0"};
{\ar_{} "-12";"10"};
{\ar^{f} "0";"2"};
{\ar^{g} "2";"4"};
{\ar^{h} "4";"6"};
{\ar_{\alpha_{X^{\prime}}} "10";"12"};
{\ar_{\beta_{X^{\prime}}} "12";"14"};
{\ar_{\gamma_{X^{\prime}}} "14";"16"};
{\ar@{-->}_{{}^{\exists}t} "4";"12"};
{\ar^{\zeta} "4";"14"};
{\ar^{0} "6";"16"};
{\ar@{}|\circlearrowright "3";"14"};
{\ar@{}|\circlearrowright "4";"16"};
\endxy
\]\end{proof}

\begin{prop}\label{PropTR3}
$(\mathcal{Z}/\mathcal{I}_{\mathcal{D}},S,\triangle)$ satisfies {\rm (TR3)}.
\end{prop}
\begin{proof}
Suppose we are given distinguished triangles
\begin{eqnarray*}
&X\overset{\underline{f}}{\longrightarrow}Y\overset{\underline{g}}{\longrightarrow}Z\overset{\underline{q}}{\longrightarrow}SX&\\
&X^{\prime}\overset{\underline{f^{\prime}}}{\longrightarrow}Y^{\prime}\overset{\underline{g^{\prime}}}{\longrightarrow}Z^{\prime}\overset{\underline{q^{\prime}}}{\longrightarrow}SX^{\prime}&
\end{eqnarray*}
and morphisms $x\in\mathcal{Z}(X,X^{\prime})$ and $y\in\mathcal{Z}(Y,Y^{\prime})$ satisfying $\underline{y}\circ\underline{f}=\underline{f^{\prime}}\circ\underline{x}$. We want to find $z\in\mathcal{Z}(Z,Z^{\prime})$ which satisfies $\underline{z}\circ\underline{g}=\underline{g^{\prime}}\circ\underline{y}$ and $S\underline{x}\circ\underline{q}=\underline{q^{\prime}}\circ\underline{z}$.

We may assume these triangles are standard, arising from morphisms of conflations:
\[
\xy
(-40,6)*+{\Omega Z}="-2";
(-24,6)*+{X}="0";
(-8,6)*+{Y}="2";
(8,6)*+{Z}="4";
(24,6)*+{\Sigma X}="6";
(-40,-6)*+{\Omega S_X}="-12";
(-24,-6)*+{X}="10";
(-8,-6)*+{I_X}="12";
(8,-6)*+{S_X}="14";
(24,-6)*+{\Sigma X}="16";
{\ar^{} "-2";"0"};
{\ar_{} "-12";"10"};
{\ar^{} "-2";"-12"};
{\ar^{f} "0";"2"};
{\ar^{g} "2";"4"};
{\ar^{h} "4";"6"};
{\ar_{\alpha_X} "10";"12"};
{\ar_{\beta_X} "12";"14"};
{\ar_{\gamma_X} "14";"16"};
{\ar@{=} "0";"10"};
{\ar^{p} "2";"12"};
{\ar^{q} "4";"14"};
{\ar@{=} "6";"16"};
{\ar@{}|\circlearrowright "0";"12"};
{\ar@{}|\circlearrowright "2";"14"};
{\ar@{}|\circlearrowright "4";"16"};
{\ar@{}|\circlearrowright "-2";"10"};
\endxy
\]
\[
\xy
(-40,6)*+{\Omega Z^{\prime}}="-2";
(-24,6)*+{X^{\prime}}="0";
(-8,6)*+{Y^{\prime}}="2";
(8,6)*+{Z^{\prime}}="4";
(24,6)*+{\Sigma X^{\prime}}="6";
(-40,-6)*+{\Omega S_{X^{\prime}}}="-12";
(-24,-6)*+{X^{\prime}}="10";
(-8,-6)*+{I_{X^{\prime}}}="12";
(8,-6)*+{S_{X^{\prime}}}="14";
(24,-6)*+{\Sigma X^{\prime}}="16";
{\ar^{f^{\prime}} "0";"2"};
{\ar^{g^{\prime}} "2";"4"};
{\ar^{h^{\prime}} "4";"6"};
{\ar_{\alpha_{X^{\prime}}} "10";"12"};
{\ar_{\beta_{X^{\prime}}} "12";"14"};
{\ar_{\gamma_{X^{\prime}}} "14";"16"};
{\ar^{} "-2";"0"};
{\ar_{} "-12";"10"};
{\ar^{} "-2";"-12"};
{\ar@{=} "0";"10"};
{\ar^{p^{\prime}} "2";"12"};
{\ar^{q^{\prime}} "4";"14"};
{\ar@{=} "6";"16"};
{\ar@{}|\circlearrowright "0";"12"};
{\ar@{}|\circlearrowright "2";"14"};
{\ar@{}|\circlearrowright "4";"16"};
{\ar@{}|\circlearrowright "-2";"10"};
\endxy
\]
Since $\underline{y}\circ\underline{f}=\underline{f^{\prime}}\circ\underline{x}$, there exist $I\in\mathcal{I}_{\mathcal{D}}$, $s_1\in\mathcal{Z}(X,I)$ and $s_2\in\mathcal{Z}(I,Y^{\prime})$ such that $s_2\circ s_1=y\circ f-f^{\prime}\circ x$. By the injectivity of $I$, there exists $s_3\in\mathcal{Z}(Y,I)$ such that $s_3\circ f=s_1$. Then we have $(y-s_2\circ s_3)\circ f=f^{\prime}\circ x$, and there exists $z\in\mathcal{Z}(Z,Z^{\prime})$ such that $z\circ g=g^{\prime}\circ(y-s_2\circ s_3)$ and $(\Sigma x)\circ h=h^{\prime}\circ z$ by {\rm (RTR3)}. Thus Proposition \ref{PropTR3} follows from Lemma \ref{LemShift}.
\[
\xy
(-40,6)*+{\Omega Z}="-2";
(-24,6)*+{X}="0";
(-8,6)*+{Y}="2";
(8,6)*+{Z}="4";
(24,6)*+{\Sigma X}="6";
(-40,-6)*+{\Omega S_X}="-12";
(-24,-6)*+{X^{\prime}}="10";
(-8,-6)*+{Y^{\prime}}="12";
(8,-6)*+{Z^{\prime}}="14";
(24,-6)*+{\Sigma X^{\prime}}="16";
{\ar^{} "-2";"0"};
{\ar_{} "-12";"10"};
{\ar@{-->}_{\Omega z} "-2";"-12"};
{\ar^{f} "0";"2"};
{\ar^{g} "2";"4"};
{\ar^{h} "4";"6"};
{\ar_{f^{\prime}} "10";"12"};
{\ar_{g^{\prime}} "12";"14"};
{\ar_{h^{\prime}} "14";"16"};
{\ar_{x} "0";"10"};
{\ar|*+{_{y-s_2\circ s_3}} "2";"12"};
{\ar@{-->}^{z} "4";"14"};
{\ar^{\Sigma x} "6";"16"};
{\ar@{}|\circlearrowright "0";"12"};
{\ar@{}|\circlearrowright "2";"14"};
{\ar@{}|\circlearrowright "4";"16"};
{\ar@{}|\circlearrowright "-2";"10"};
\endxy
\]
\end{proof}

\begin{prop}\label{PropTR4}
$(\mathcal{Z}/\mathcal{I}_{\mathcal{D}},S,\triangle)$ satisfies {\rm (TR4)}.
\end{prop}
\begin{proof}
Let
\begin{eqnarray}
X\overset{\underline{\ell}}{\longrightarrow}M\overset{\underline{m}}{\longrightarrow}Y^{\prime}\overset{\underline{v}}{\longrightarrow}SX \label{triangle2}\\
X^{\prime}\overset{\underline{\ell^{\prime}}}{\longrightarrow}M\overset{\underline{m^{\prime}}}{\longrightarrow}Y\overset{\underline{v^{\prime}}}{\longrightarrow}SX^{\prime} \label{triangle3}\\
X\overset{\underline{f}}{\longrightarrow}Y\overset{\underline{g}}{\longrightarrow}Z\overset{\underline{q}}{\longrightarrow}SX, \label{triangle1}
\end{eqnarray}
be distinguished triangles in $\mathcal{Z}/\mathcal{I}_{\mathcal{D}}$ satisfying $\underline{m^{\prime}}\circ\underline{\ell}=\underline{f}$.
It suffices to show there exist $g^{\prime}\in\mathcal{Z}(Y^{\prime},Z)$ and $q^{\prime}\in\mathcal{Z}(Z,S_{X^{\prime}})$ such that
\[ X^\prime\overset{\underline{f^{\prime}}}{\longrightarrow}Y^\prime\overset{\underline{g^{\prime}}}{\longrightarrow}Z\overset{\underline{q^{\prime}}}{\longrightarrow}SX^\prime \]
is a standard triangle, where $f^{\prime}=m\circ\ell^{\prime}$, and satisfy
\begin{eqnarray*}
\underline{g^{\prime}}\circ\underline{m}=\underline{g}\circ\underline{m^{\prime}}&,&\underline{q^{\prime}}\circ\underline{g}=\underline{v^{\prime}},\\
\underline{q}\circ\underline{g^{\prime}}=\underline{v}&,&S\underline{\ell^{\prime}}\circ\underline{q^{\prime}}+S\underline{\ell}\circ\underline{q}=0.
\end{eqnarray*}

\[
\xy
(-30,10)*+{X}="2";
(-20,0)*+{M}="4";
(-10,-10)*+{Y^{\prime}}="6";
(10,-10)*+{SX}="8";
(10,10)*+{SX^{\prime}}="10";
(-10,10)*+{Y}="12";
(0,0)*+{Z}="14";
(-30,-10)*+{X^{\prime}}="18";
(20,0)*+{SM}="20";
(-20,14)*+{}="22";
(-20,-14)*+{}="24";
(0,14)*+{}="26";
(0,-14)*+{}="28";
{\ar_{\underline{\ell}} "2";"4"};
{\ar^{\underline{m}} "4";"6"};
{\ar_{\underline{v}} "6";"8"};
{\ar^{\underline{f}} "2";"12"};
{\ar_{\underline{g}} "12";"14"};
{\ar^{\underline{q}} "14";"8"};
{\ar^{\underline{\ell^{\prime}}} "18";"4"};
{\ar_{\underline{m^{\prime}}} "4";"12"};
{\ar^{\underline{v^{\prime}}} "12";"10"};
{\ar^{S\underline{\ell^{\prime}}} "10";"20"};
{\ar_{-S\underline{\ell}} "8";"20"};
{\ar@{-->}^{\underline{g^{\prime}}} "6";"14"};
{\ar_{\underline{f^{\prime}}} "18";"6"};
{\ar@{-->}_{\underline{q^{\prime}}} "14";"10"};
{\ar@{}|\circlearrowright "6";"12"};
{\ar@{}|\circlearrowright "4";"22"};
{\ar@{}|\circlearrowright "4";"24"};
{\ar@{}|\circlearrowright "14";"26"};
{\ar@{}|\circlearrowright "14";"28"};
{\ar@{}|\circlearrowright "8";"10"};
\endxy
\]

We may assume {\rm (\ref{triangle2}), (\ref{triangle3}), (\ref{triangle1})} are standard triangles, arising from the following morphisms of conflations.
\[
\xy
(-24,6)*+{\Omega Y^{\prime}}="-2";
(-12,6)*+{X}="0";
(0,6)*+{M}="2";
(12,6)*+{Y^{\prime}}="4";
(24,6)*+{\Sigma X}="6";
(-12,-6)*+{X}="10";
(0,-6)*+{I_X}="12";
(12,-6)*+{S_X}="14";
(24,-6)*+{\Sigma X}="16";
{\ar^{} "-2";"0"};
{\ar^{\ell} "0";"2"};
{\ar^{m} "2";"4"};
{\ar^{n} "4";"6"};
(-24,-6)*+{\Omega S_X}="-12";
{\ar_{\alpha_X} "10";"12"};
{\ar_{\beta_X} "12";"14"};
{\ar_{\gamma_X} "14";"16"};
{\ar@{=} "0";"10"};
{\ar^{u} "2";"12"};
{\ar^{v} "4";"14"};
{\ar@{=} "6";"16"};
{\ar^{} "-12";"10"};
{\ar^{} "-2";"-12"};
{\ar@{}|\circlearrowright "-2";"10"};
{\ar@{}|\circlearrowright "0";"12"};
{\ar@{}|\circlearrowright "2";"14"};
{\ar@{}|\circlearrowright "4";"16"};
\endxy
,
\xy
(-24,6)*+{\Omega Y}="-2";
(-12,6)*+{X^{\prime}}="0";
(0,6)*+{M}="2";
(12,6)*+{Y}="4";
(24,6)*+{\Sigma X^{\prime}}="6";
(-12,-6)*+{X^{\prime}}="10";
(-24,-6)*+{\Omega S_{X^{\prime}}}="-12";
(0,-6)*+{I_{X^{\prime}}}="12";
(12,-6)*+{S_{X^{\prime}}}="14";
(24,-6)*+{\Sigma X^{\prime}}="16";
{\ar^{\ell^{\prime}} "0";"2"};
{\ar^{m^{\prime}} "2";"4"};
{\ar^{n^{\prime}} "4";"6"};
{\ar_{\alpha_{X^{\prime}}} "10";"12"};
{\ar_{\beta_{X^{\prime}}} "12";"14"};
{\ar_{\gamma_{X^{\prime}}} "14";"16"};
{\ar^{} "-2";"0"};
{\ar@{=} "0";"10"};
{\ar^{u^{\prime}} "2";"12"};
{\ar^{v^{\prime}} "4";"14"};
{\ar@{=} "6";"16"};
{\ar^{} "-12";"10"};
{\ar^{} "-2";"-12"};
{\ar@{}|\circlearrowright "-2";"10"};
{\ar@{}|\circlearrowright "0";"12"};
{\ar@{}|\circlearrowright "2";"14"};
{\ar@{}|\circlearrowright "4";"16"};
\endxy
\]
\begin{equation}
\xy
(-24,6)*+{\Omega Z}="-2";
(-12,6)*+{X}="0";
(0,6)*+{Y}="2";
(12,6)*+{Z}="4";
(24,6)*+{\Sigma X}="6";
(-24,-6)*+{\Omega S_X}="-12";
(-12,-6)*+{X}="10";
(0,-6)*+{I_X}="12";
(12,-6)*+{S_X}="14";
(24,-6)*+{\Sigma X}="16";
{\ar^{f} "0";"2"};
{\ar^{g} "2";"4"};
{\ar^{h} "4";"6"};
{\ar_{\alpha_X} "10";"12"};
{\ar_{\beta_X} "12";"14"};
{\ar_{\gamma_X} "14";"16"};
{\ar^{} "-2";"0"};
{\ar@{=} "0";"10"};
{\ar^{p} "2";"12"};
{\ar^{q} "4";"14"};
{\ar@{=} "6";"16"};
{\ar^{} "-12";"10"};
{\ar^{} "-2";"-12"};
{\ar@{}|\circlearrowright "-2";"10"};
{\ar@{}|\circlearrowright "0";"12"};
{\ar@{}|\circlearrowright "2";"14"};
{\ar@{}|\circlearrowright "4";"16"};
\endxy
\label{ConfDiagLast}
\end{equation}

\begin{claim}\label{ClaimBefLas}
We may assume $m^{\prime}\circ\ell=f$.
\end{claim}
\begin{proof}[Proof of Claim \ref{ClaimBefLas}]
Since $\underline{m^{\prime}}\circ\underline{\ell}=\underline{f}$, there exist $I\in\mathcal{I}_{\mathcal{D}}$, $f_1\in\mathcal{Z}(X,I)$ and $f_2\in\mathcal{Z}(I,Y)$ such that $f_2\circ f_1=f-m^{\prime}\circ\ell$.
Let $i_M\colon M\rightarrow M\oplus I$ and $p_M\colon M\oplus I\rightarrow M$ be the inclusion and the projection, respectively.
By Corollary \ref{CorExtension} and Lemma \ref{LemSPL}, we have extensions
\begin{eqnarray*}
\Omega Q\rightarrow X\overset{(\ell,f_1)}{\longrightarrow}M\oplus I\rightarrow Q\rightarrow \Sigma X,\\
\Omega M\rightarrow I\rightarrow M\oplus I\overset{p_M}{\longrightarrow}M\rightarrow \Sigma I.
\end{eqnarray*}
By Proposition \ref{PropOCT}, we obtain the following morphisms of extensions by Lemma \ref{LemShift}.
\[
\xy
(-32,0)*+{\Omega Q}="0";
(-16,0)*+{X}="2";
(0,0)*+{M\oplus I}="4";
(16,0)*+{Q}="6";
(32,0)*+{\Sigma X}="8";
(0,24)*+{\Sigma I}="10";
(0,12)*+{M}="12";
(10,20)*+{}="13";
(16,12)*+{Y^{\prime}}="14";
(-16,-12)*+{\Omega Y^{\prime}}="16";
(0,-12)*+{I}="18";
(0,-24)*+{\Omega M}="20";
(-10,8)*+{}="22";
(10,-8)*+{}="24";
(26,8)*+{}="26";
{\ar^{} "0";"2"};
{\ar_{(\ell,f_1)} "2";"4"};
{\ar_{} "4";"6"};
{\ar_{} "6";"8"};
{\ar^{} "16";"2"};
{\ar^{\ell} "2";"12"};
{\ar_{m} "12";"14"};
{\ar^{n} "14";"8"};
{\ar^{} "20";"18"};
{\ar^{} "18";"4"};
{\ar|*+{_{p_M}} "4";"12"};
{\ar^{} "12";"10"};
{\ar@{-->}^{{}^{\exists}\rho} "6";"14"};
{\ar@{-->}_{} "16";"18"};
{\ar@{-->}_{} "18";"6"};
{\ar@{-->}_{} "14";"10"};
{\ar@{}|\circlearrowright "6";"12"};
{\ar@{}|\circlearrowright "4";"22"};
{\ar@{}|\circlearrowright "4";"24"};
{\ar@{}|\circlearrowright "6";"26"};
{\ar@{}|\circlearrowright "12";"13"};
\endxy
\]
Thus we have $Q\in\mathcal{Z}$, and obtain an isomorphism of distinguished triangles:
\[
\xy
(-26,6)*+{X}="0";
(-8,6)*+{M\oplus I}="2";
(8,6)*+{Q}="4";
(24,6)*+{SX}="6";
(-26,-6)*+{X}="10";
(-8,-6)*+{M}="12";
(8,-6)*+{Y^{\prime}}="14";
(24,-6)*+{SX}="16";
{\ar^{\underline{(\ell,f_1)}} "0";"2"};
{\ar^{} "2";"4"};
{\ar^{} "4";"6"};
{\ar_{\underline{\ell}} "10";"12"};
{\ar_{\underline{m}} "12";"14"};
{\ar_{} "14";"16"};
{\ar@{=} "0";"10"};
{\ar^{\underline{p_M}}_{\cong} "2";"12"};
{\ar^{\underline{\rho}}_{\cong} "4";"14"};
{\ar@{=} "6";"16"};
{\ar@{}|\circlearrowright "0";"12"};
{\ar@{}|\circlearrowright "2";"14"};
{\ar@{}|\circlearrowright "4";"16"};
\endxy
\]
Dually, there exist morphisms of extensions
\[
\xy
(-32,0)*+{\Sigma R}="0";
(-16,0)*+{Y}="2";
(0,0)*+{M\oplus I}="4";
(16,0)*+{{}^{\exists}R}="6";
(32,0)*+{\Omega X^{\prime}}="8";
(0,24)*+{\Omega I}="10";
(0,12)*+{M}="12";
(10,20)*+{}="13";
(16,12)*+{X^{\prime}}="14";
(-16,-12)*+{\Sigma X^{\prime}}="16";
(0,-12)*+{I}="18";
(0,-24)*+{\Sigma M}="20";
(-10,8)*+{}="22";
(10,-8)*+{}="24";
(26,8)*+{}="26";
{\ar^{} "2";"0"};
{\ar^{m^{\prime}+f_2} "4";"2"};
{\ar_{} "6";"4"};
{\ar_{} "8";"6"};
{\ar^{n^{\prime}} "2";"16"};
{\ar_{m^{\prime}} "12";"2"};
{\ar_{\ell^{\prime}} "14";"12"};
{\ar^{} "8";"14"};
{\ar^{} "18";"20"};
{\ar^{} "4";"18"};
{\ar|*+{_{i_M}} "12";"4"};
{\ar^{} "10";"12"};
{\ar@{-->}_{{}^{\exists}\omega} "14";"6"};
{\ar@{-->}_{} "18";"16"};
{\ar@{-->}_{} "6";"18"};
{\ar@{-->}_{} "10";"14"};
{\ar@{}|\circlearrowright "6";"12"};
{\ar@{}|\circlearrowright "4";"22"};
{\ar@{}|\circlearrowright "4";"24"};
{\ar@{}|\circlearrowright "6";"26"};
{\ar@{}|\circlearrowright "12";"13"};
\endxy
\]
which implies $R\in\mathcal{Z}$ and yields an isomorphism of distinguished triangles
\[
\xy
(-24,6)*+{X^{\prime}}="0";
(-9,6)*+{M}="2";
(9,6)*+{Y}="4";
(24,6)*+{SX^{\prime}}="6";
(-24,-6)*+{R}="10";
(-9,-6)*+{M\oplus I}="12";
(9,-6)*+{Y}="14";
(24,-6)*+{SR}="16";
{\ar^{\underline{\ell^{\prime}}} "0";"2"};
{\ar^{\underline{m^{\prime}}} "2";"4"};
{\ar^{} "4";"6"};
{\ar_{} "10";"12"};
{\ar_{\underline{m^{\prime}+f_2}} "12";"14"};
{\ar_{} "14";"16"};
{\ar^{\underline{\omega}}_{\cong} "0";"10"};
{\ar_{\cong}^{\underline{i_M}} "2";"12"};
{\ar@{=} "4";"14"};
{\ar^{S\underline{\omega}}_{\cong} "6";"16"};
{\ar@{}|\circlearrowright "0";"12"};
{\ar@{}|\circlearrowright "2";"14"};
{\ar@{}|\circlearrowright "4";"16"};
\endxy
\]
Thus, replacing $\ell$ by $(\ell,f_1)$ and $m^{\prime}$ by $m^{\prime}+f_2$, we may assume $m^{\prime}\circ\ell=f$.
\end{proof}

By Claim \ref{ClaimBefLas}, assume $m^{\prime}\circ\ell=f$.
Then by Proposition \ref{PropOCT}, there exist $g^{\prime}\in\mathcal{Z}(Y^{\prime},Z)$ and $h^{\prime}\in\mathcal{C}(Z,\Sigma X^{\prime})$ such that
\[ \Omega Z\rightarrow X^{\prime}\overset{f^{\prime}}{\longrightarrow}Y^{\prime}\overset{g^{\prime}}{\longrightarrow}Z\overset{h^{\prime}}{\longrightarrow}\Sigma X^{\prime} \]
is a conflation, and make the following diagram commutative.
\[
\xy
(-30,10)*+{X}="2";
(-20,0)*+{M}="4";
(-10,-10)*+{Y^{\prime}}="6";
(10,-10)*+{\Sigma X}="8";
(10,10)*+{\Sigma X^{\prime}}="10";
(-10,10)*+{Y}="12";
(0,0)*+{Z}="14";
(-30,-10)*+{X^{\prime}}="18";
(20,0)*+{\Sigma M}="20";
(-20,14)*+{}="22";
(-20,-14)*+{}="24";
(0,14)*+{}="26";
(0,-14)*+{}="28";
{\ar_{\ell} "2";"4"};
{\ar^{m} "4";"6"};
{\ar_{n} "6";"8"};
{\ar^{f} "2";"12"};
{\ar_{g} "12";"14"};
{\ar^{h} "14";"8"};
{\ar^{\ell^{\prime}} "18";"4"};
{\ar_{m^{\prime}} "4";"12"};
{\ar^{n^{\prime}} "12";"10"};
{\ar^{\Sigma\ell^{\prime}} "10";"20"};
{\ar_{-\Sigma\ell} "8";"20"};
{\ar@{-->}^{g^{\prime}} "6";"14"};
{\ar_{f^{\prime}} "18";"6"};
{\ar@{-->}_{h^{\prime}} "14";"10"};
{\ar@{}|\circlearrowright "6";"12"};
{\ar@{}|\circlearrowright "4";"22"};
{\ar@{}|\circlearrowright "4";"24"};
{\ar@{}|\circlearrowright "14";"26"};
{\ar@{}|\circlearrowright "14";"28"};
{\ar@{}|\circlearrowright "8";"10"};
\endxy
\]
If we take a morphism of conflations
\[
\xy
(-24,6)*+{\Omega Z}="-2";
(-12,6)*+{X^{\prime}}="0";
(0,6)*+{Y^{\prime}}="2";
(12,6)*+{Z}="4";
(24,6)*+{\Sigma X^{\prime}}="6";
(-24,-6)*+{\Omega S_{X^{\prime}}}="-12";
(-12,-6)*+{X^{\prime}}="10";
(0,-6)*+{I_{X^{\prime}}}="12";
(12,-6)*+{S_{X^{\prime}}}="14";
(24,-6)*+{\Sigma X^{\prime}}="16";
{\ar^{f^{\prime}} "0";"2"};
{\ar^{g^{\prime}} "2";"4"};
{\ar^{h^{\prime}} "4";"6"};
{\ar_{\alpha_{X^{\prime}}} "10";"12"};
{\ar_{\beta_{X^{\prime}}} "12";"14"};
{\ar_{\gamma_{X^{\prime}}} "14";"16"};
{\ar^{} "-2";"0"};
{\ar@{=} "0";"10"};
{\ar^{p^{\prime}} "2";"12"};
{\ar^{q^{\prime}} "4";"14"};
{\ar@{=} "6";"16"};
{\ar^{} "-12";"10"};
{\ar^{} "-2";"-12"};
{\ar@{}|\circlearrowright "-2";"10"};
{\ar@{}|\circlearrowright "0";"12"};
{\ar@{}|\circlearrowright "2";"14"};
{\ar@{}|\circlearrowright "4";"16"};
\endxy
\]
then by Lemma \ref{LemShift}, we obtain morphisms of standard triangles
\[
\xy
(-12,6)*+{X^{\prime}}="0";
(0,6)*+{M}="2";
(12,6)*+{Y}="4";
(24,6)*+{SX^{\prime}}="6";
(-12,-6)*+{X^{\prime}}="10";
(0,-6)*+{Y^{\prime}}="12";
(12,-6)*+{Z}="14";
(24,-6)*+{SX^{\prime}}="16";
{\ar^{\underline{\ell^{\prime}}} "0";"2"};
{\ar^{\underline{m^{\prime}}} "2";"4"};
{\ar^{\underline{v^{\prime}}} "4";"6"};
{\ar_{\underline{f^{\prime}}} "10";"12"};
{\ar_{\underline{g^{\prime}}} "12";"14"};
{\ar_{\underline{q^{\prime}}} "14";"16"};
{\ar@{=} "0";"10"};
{\ar^{\underline{m}} "2";"12"};
{\ar^{\underline{g}} "4";"14"};
{\ar@{=} "6";"16"};
{\ar@{}|\circlearrowright "0";"12"};
{\ar@{}|\circlearrowright "2";"14"};
{\ar@{}|\circlearrowright "4";"16"};
\endxy
\ \ \text{and}\ \ 
\xy
(-12,6)*+{X}="0";
(0,6)*+{M}="2";
(12,6)*+{Y^{\prime}}="4";
(24,6)*+{SX}="6";
(-12,-6)*+{X}="10";
(0,-6)*+{Y}="12";
(12,-6)*+{Z}="14";
(24,-6)*+{SX}="16";
{\ar^{\underline{\ell}} "0";"2"};
{\ar^{\underline{m}} "2";"4"};
{\ar^{\underline{v}} "4";"6"};
{\ar_{\underline{f}} "10";"12"};
{\ar_{\underline{g}} "12";"14"};
{\ar_{\underline{q}} "14";"16"};
{\ar@{=} "0";"10"};
{\ar^{\underline{m^{\prime}}} "2";"12"};
{\ar^{\underline{g^{\prime}}} "4";"14"};
{\ar@{=} "6";"16"};
{\ar@{}|\circlearrowright "0";"12"};
{\ar@{}|\circlearrowright "2";"14"};
{\ar@{}|\circlearrowright "4";"16"};
\endxy
\]

Thus it remains to show $S\underline{\ell^{\prime}}\circ\underline{q^{\prime}}+S\underline{\ell}\circ\underline{q}=0$.

\begin{claim}\label{ClaimForPropETR4}
There exist morphisms of conflations
\begin{eqnarray}
&\xy
(-24,6)*+{\Omega Z}="-2";
(-12,6)*+{X}="0";
(0,6)*+{Y}="2";
(12,6)*+{Z}="4";
(24,6)*+{\Sigma X}="6";
(-24,-6)*+{\Omega S_M}="-12";
(-12,-6)*+{M}="10";
(0,-6)*+{I_M}="12";
(12,-6)*+{S_M}="14";
(24,-6)*+{\Sigma M}="16";
{\ar^{f} "0";"2"};
{\ar^{g} "2";"4"};
{\ar^{h} "4";"6"};
{\ar_{\alpha_M} "10";"12"};
{\ar_{\beta_M} "12";"14"};
{\ar_{\gamma_M} "14";"16"};
{\ar^{} "-2";"0"};
{\ar^{\ell} "0";"10"};
{\ar^{r} "2";"12"};
{\ar^{s} "4";"14"};
{\ar^{\Sigma\ell} "6";"16"};
{\ar^{} "-12";"10"};
{\ar^{} "-2";"-12"};
{\ar@{}|\circlearrowright "-2";"10"};
{\ar@{}|\circlearrowright "0";"12"};
{\ar@{}|\circlearrowright "2";"14"};
{\ar@{}|\circlearrowright "4";"16"};
\endxy&\label{DoubleConf1}\\ 
&\xy
(-24,6)*+{\Omega Z}="-2";
(-12,6)*+{X^{\prime}}="0";
(0,6)*+{Y^{\prime}}="2";
(12,6)*+{Z}="4";
(24,6)*+{\Sigma X^{\prime}}="6";
(-24,-6)*+{\Omega S_M}="-12";
(-12,-6)*+{M}="10";
(0,-6)*+{I_M}="12";
(12,-6)*+{S_M}="14";
(24,-6)*+{\Sigma M}="16";
{\ar^{f^{\prime}} "0";"2"};
{\ar^{g^{\prime}} "2";"4"};
{\ar^{h^{\prime}} "4";"6"};
{\ar_{\alpha_M} "10";"12"};
{\ar_{\beta_M} "12";"14"};
{\ar_{\gamma_M} "14";"16"};
{\ar^{} "-2";"0"};
{\ar^{\ell^{\prime}} "0";"10"};
{\ar^{r^{\prime}} "2";"12"};
{\ar^{s^{\prime}} "4";"14"};
{\ar^{\Sigma\ell^{\prime}} "6";"16"};
{\ar^{} "-12";"10"};
{\ar^{} "-2";"-12"};
{\ar@{}|\circlearrowright "-2";"10"};
{\ar@{}|\circlearrowright "0";"12"};
{\ar@{}|\circlearrowright "2";"14"};
{\ar@{}|\circlearrowright "4";"16"};
\endxy&
\label{DoubleConf2}
\end{eqnarray}
such that
\[ r\circ m^{\prime}+r^{\prime}\circ m=\alpha_M. \]
Moreover, $s$ and $s^{\prime}$ satisfy
\begin{eqnarray}
\underline{s}=S\underline{\ell}\circ\underline{q}%
\quad\text{}and\quad
\underline{s^{\prime}}=S\underline{\ell^{\prime}}\circ\underline{q^{\prime}}.
\label{LastEq}
\end{eqnarray}
\end{claim}
Suppose Claim \ref{ClaimForPropETR4} is shown. Then by
\begin{eqnarray*}
(s+s^{\prime})\circ g\circ m^{\prime}&=&s\circ g\circ m^{\prime}+s^{\prime}\circ g^{\prime}\circ m\\
&=&\beta_M\circ r\circ m^{\prime}+\beta_M\circ r^{\prime}\circ m\\
&=&\beta_M\circ\alpha_M=0,
\end{eqnarray*}
there exists $w^{\prime}\in\mathcal{C}(\Sigma X^{\prime},S_M)$ such that $w^{\prime}\circ n^{\prime}=(s+s^{\prime})\circ g$. Thus by $((s+s^{\prime})-w^{\prime}\circ h^{\prime})\circ g=0$, there exists $w\in\mathcal{C}(\Sigma X,S_M)$ such that $w\circ h=s+s^{\prime}-w^{\prime}\circ h^{\prime}$, namely
\[ s+s^{\prime}=w\circ h+w^{\prime}\circ h^{\prime}. \]
Take a conflation
\[ \Omega Z\rightarrow X_0\rightarrow I_0\overset{\beta_0}{\longrightarrow}Z\overset{\gamma_0}{\longrightarrow}\Sigma X_0 \]
with $I_0\in\mathcal{I}_{\mathcal{D}}$. We have morphisms of conflations
\[
\xy
(-24,6)*+{\Omega Z}="-2";
(-12,6)*+{X_0}="0";
(0,6)*+{I_0}="2";
(12,6)*+{Z}="4";
(24,6)*+{\Sigma X_0}="6";
(-24,-6)*+{\Omega Z}="-12";
(-12,-6)*+{X}="10";
(0,-6)*+{Y}="12";
(12,-6)*+{Z}="14";
(24,-6)*+{\Sigma X}="16";
{\ar^{} "-2";"0"};
{\ar^{} "0";"2"};
{\ar^{\beta_0} "2";"4"};
{\ar^{\gamma_0} "4";"6"};
{\ar_{f} "10";"12"};
{\ar_{g} "12";"14"};
{\ar_{h} "14";"16"};
{\ar^{} "-2";"0"};
{\ar@{-->} "0";"10"};
{\ar@{-->} "2";"12"};
{\ar@{=} "4";"14"};
{\ar@{-->}^{{}^{\exists}\xi} "6";"16"};
{\ar^{} "-12";"10"};
{\ar@{=} "-2";"-12"};
{\ar@{}|\circlearrowright "-2";"10"};
{\ar@{}|\circlearrowright "0";"12"};
{\ar@{}|\circlearrowright "2";"14"};
{\ar@{}|\circlearrowright "4";"16"};
\endxy
\ ,\ 
\xy
(-24,6)*+{\Omega Z}="-2";
(-12,6)*+{X_0}="0";
(0,6)*+{I_0}="2";
(12,6)*+{Z}="4";
(24,6)*+{\Sigma X_0}="6";
(-24,-6)*+{\Omega Z}="-12";
(-12,-6)*+{X^{\prime}}="10";
(0,-6)*+{Y^{\prime}}="12";
(12,-6)*+{Z}="14";
(24,-6)*+{\Sigma X^{\prime}}="16";
{\ar^{} "-2";"0"};
{\ar^{} "0";"2"};
{\ar^{\beta_0} "2";"4"};
{\ar^{\gamma_0} "4";"6"};
{\ar_{f^{\prime}} "10";"12"};
{\ar_{g^{\prime}} "12";"14"};
{\ar_{h^{\prime}} "14";"16"};
{\ar^{} "-2";"0"};
{\ar@{-->} "0";"10"};
{\ar@{-->} "2";"12"};
{\ar@{=} "4";"14"};
{\ar@{-->}^{{}^{\exists}\xi^{\prime}} "6";"16"};
{\ar^{} "-12";"10"};
{\ar@{=} "-2";"-12"};
{\ar@{}|\circlearrowright "-2";"10"};
{\ar@{}|\circlearrowright "0";"12"};
{\ar@{}|\circlearrowright "2";"14"};
{\ar@{}|\circlearrowright "4";"16"};
\endxy
,
\]
and thus obtain
\[ s+s^{\prime}=(w\circ\xi+w^{\prime}\circ\xi^{\prime})\circ\gamma_0. \]
Since $\gamma_M\circ (s+s^{\prime})=(\Sigma\ell)\circ h+(\Sigma\ell^{\prime})\circ h^{\prime}=0$, we can conclude that $s+s^{\prime}$ factors through $I_M$ by {\rm (AC1)}.
\[
\xy
(-24,6)*+{\Omega Z}="-2";
(-12,6)*+{X_0}="0";
(0,6)*+{I_0}="2";
(4,2)*+{}="3";
(12,6)*+{Z}="4";
(24,6)*+{\Sigma X_0}="6";
(-24,-6)*+{\Omega S_M}="-12";
(-12,-6)*+{M}="10";
(0,-6)*+{I_M}="12";
(12,-6)*+{S_M}="14";
(24,-6)*+{\Sigma M}="16";
{\ar^{} "-2";"0"};
{\ar^{} "0";"2"};
{\ar^{\beta_0} "2";"4"};
{\ar^{\gamma_0} "4";"6"};
{\ar_{\alpha_M} "10";"12"};
{\ar_{\beta_M} "12";"14"};
{\ar_{\gamma_M} "14";"16"};
{\ar^{} "-2";"0"};
{\ar^{s+s^{\prime}} "4";"14"};
{\ar^{} "-12";"10"};
{\ar@{-->}^{} "4";"12"};
{\ar@{}|\circlearrowright "3";"14"};
\endxy
\]
By {\rm (\ref{LastEq})}, this means $S\underline{\ell^{\prime}}\circ\underline{q^{\prime}}+S\underline{\ell}\circ\underline{q}=0$, and Proposition \ref{PropTR4} can be shown.
Thus it suffices to show Claim \ref{ClaimForPropETR4}.
\begin{proof}[Proof of Claim \ref{ClaimForPropETR4}]
By $I_M\in\mathcal{I}_{\mathcal{D}}$, there exists $r\in\mathcal{Z}(Y,I_M)$ such that $r\circ f=\alpha_M\circ\ell$. By $(\alpha_M-r\circ m^{\prime})\circ\ell=0$, there exists $r^{\prime}\in\mathcal{Z}(Y^{\prime},I_M)$ such that $r^{\prime}\circ m=\alpha_M-r\circ m^{\prime}$. By {\rm (RTR3)}, there exist $s,s^{\prime}\in\mathcal{Z}(Z,S_M)$ such that {\rm (\ref{DoubleConf1})} and {\rm (\ref{DoubleConf2})} are morphisms of conflations.

By definition, $S_{\ell}$ is a morphism which gives a morphism of conflations as follows.
\[
\xy
(-24,6)*+{\Omega S_X}="-2";
(-12,6)*+{X}="0";
(0,6)*+{I_X}="2";
(12,6)*+{S_X}="4";
(24,6)*+{\Sigma X}="6";
(-24,-6)*+{\Omega S_M}="-12";
(-12,-6)*+{M}="10";
(0,-6)*+{I_M}="12";
(12,-6)*+{S_M}="14";
(24,-6)*+{\Sigma M}="16";
{\ar^{} "-2";"0"};
{\ar^{\alpha_X} "0";"2"};
{\ar^{\beta_X} "2";"4"};
{\ar^{\gamma_X} "4";"6"};
{\ar_{\alpha_M} "10";"12"};
{\ar_{\beta_M} "12";"14"};
{\ar_{\gamma_M} "14";"16"};
{\ar^{} "-2";"0"};
{\ar^{\ell} "0";"10"};
{\ar^{I_{\ell}} "2";"12"};
{\ar^{S_{\ell}} "4";"14"};
{\ar^{\Sigma\ell} "6";"16"};
{\ar^{} "-12";"10"};
{\ar^{} "-2";"-12"};
{\ar@{}|\circlearrowright "-2";"10"};
{\ar@{}|\circlearrowright "0";"12"};
{\ar@{}|\circlearrowright "2";"14"};
{\ar@{}|\circlearrowright "4";"16"};
\endxy
\]
Composing with {\rm (\ref{ConfDiagLast})}, we obtain a morphism of conflations
\[
\xy
(-24,6)*+{\Omega Z}="-2";
(-12,6)*+{X}="0";
(0,6)*+{Y}="2";
(12,6)*+{Z}="4";
(24,6)*+{\Sigma X}="6";
(-24,-6)*+{\Omega S_M}="-12";
(-12,-6)*+{M}="10";
(0,-6)*+{I_M}="12";
(12,-6)*+{S_M}="14";
(24,-6)*+{\Sigma M}="16";
{\ar^{f} "0";"2"};
{\ar^{g} "2";"4"};
{\ar^{h} "4";"6"};
{\ar_{\alpha_M} "10";"12"};
{\ar_{\beta_M} "12";"14"};
{\ar_{\gamma_M} "14";"16"};
{\ar^{} "-2";"0"};
{\ar^{\ell} "0";"10"};
{\ar|{_{I_{\ell}\circ p}} "2";"12"};
{\ar|{_{S_{\ell}\circ q}} "4";"14"};
{\ar^{\Sigma\ell} "6";"16"};
{\ar^{} "-12";"10"};
{\ar^{} "-2";"-12"};
{\ar@{}|\circlearrowright "-2";"10"};
{\ar@{}|\circlearrowright "0";"12"};
{\ar@{}|\circlearrowright "2";"14"};
{\ar@{}|\circlearrowright "4";"16"};
\endxy
\]
Thus, comparing with {\rm (\ref{DoubleConf1})}, we obtain $\underline{s}=S\underline{\ell}\circ \underline{q}$ by Lemma \ref{LemForSuspension}. Similarly for $s^{\prime}$.
\end{proof}
\end{proof}

By the above arguments, we obtain the following.
\begin{thm}\label{MainThm}
Let $\mathcal{C}$ be a pseudo-triangulated category satisfying Condition \ref{CondPTR}, and let $\mathcal{Z}\subseteq\mathcal{C}$ be an extension-closed subcategory, and let $\mathcal{D}\subseteq\mathcal{Z}$ is a full additive replete subcategory closed under finite direct summands in $\mathcal{Z}$. If $(\mathcal{C},\mathcal{Z},\mathcal{D})$ is Frobenius, then $\mathcal{Z}/\mathcal{I}_{\mathcal{D}}$ becomes a triangulated category.

In particular, if $\mathcal{Z}$ is Frobenius, then the stable category $\mathcal{Z}/\mathcal{I}$ becomes a triangulated category.
\end{thm}

\section{Possibility of further generalizations}

In \cite{B}, for any triangulated category $\mathcal{C}$, Beligiannis showed that if we are given a {\it proper class of triangles} $\mathcal{E}$ on $\mathcal{C}$ satisfying some conditions similar to the Frobenius condition discussed in section \ref{SecFrobCond}, then $\mathcal{C}/\mathcal{P}(\mathcal{E})$ becomes triangulated (Theorem 7.2 in \cite{B}). Here, $\mathcal{P}(\mathcal{E})$ is the subcategory of \lq projectives', defined in a similar, but different manner (Definition 4.1 in \cite{B}). With that definition, $\mathcal{P}(\mathcal{\mathcal{E}})$ becomes closed under $\Sigma$, but this conflicts with Iyama-Yoshino's construction, in which the factoring category $\mathcal{D}$ satisfies $\mathcal{C}(\mathcal{D},\Sigma\mathcal{D})=0$. We wonder if there exists a general construction unifying the construction in \cite{B} and that in section \ref{SecMainThm}.

We also remark that there is another very general construction of a triangulated stable category. In \cite{BM}, Beligiannis and Marmaridis constructed a left triangulated category (in the sense of \cite{B} or \cite{BM}) from a pair $(\mathcal{C},\mathcal{X})$ of an additive category $\mathcal{C}$ and a contravariantly finite subcategory $\mathcal{X}$ assuming some existence condition on kernels (Theorem 2.12 in \cite{BM}). Therefore if $\mathcal{X}$ is functorially finite and satisfies some nice properties, it is expected that this resulting category becomes triangulated. In fact, Happel's construction is one of these cases (Remark 2.14 in \cite{BM}).
Although this existence condition is not satisfied by a triangulated category $\mathcal{C}$ unless we replace it by some \lq pseudo' one, we hope some unifying construction will be possible.

\end{document}